\documentclass[12pt,draft]{amsart}
\usepackage{amscd,amssymb,amsmath}
\usepackage[all]{xy}
\usepackage{amsfonts}
\usepackage{amssymb}
\usepackage{epic}
\usepackage{color}

\newtheorem{theorem}{Theorem}[section]
\newtheorem{lemma}[theorem]{Lemma}
\theoremstyle{definition}

\newtheorem{corollary}[theorem]{Corollary}

\newtheorem{conjecture}[theorem]{Conjecture}
\theoremstyle{remark}
\newtheorem{remark}{Remark}[section]
\newtheorem{teo}{Theorem}[section]

\theoremstyle{definition}
\newtheorem{dfn}[teo]{Definition}
\newtheorem{rk}[teo]{Remark}
\newtheorem{ex}[teo]{Example}
\newtheorem{que}[teo]{Question}
\numberwithin{equation}{section}

\newcommand{\wh}{\widehat}

\newcommand{\N}{\mathcal N}
\newcommand{\R}{\mathcal R}
\newcommand{\ga}{\Gamma}

\newcommand{\bea} {\begin{eqnarray*}}
\newcommand{\beq} {\begin{equation}}
\newcommand{\bey} {\begin{eqnarray}}
\newcommand{\eea} {\end{eqnarray*}}
\newcommand{\eeq} {\end{equation}}
\newcommand{\eey} {\end{eqnarray}}
\newcommand{\til}[1]{\widetilde{#1}}

\def\Z{{\mathbb Z}}

\bibliography{ref}
\def\Fix{\operatorname{Fix}}

\def\N{{\mathbb N}}

\def\<{\langle}
\def\>{\rangle}

\def\r{\rho}

\def\w{\omega}

\def\G{{\Gamma}}

\def\C{{\mathbb C}}
\def\R{{\mathbb R}}
\def\Z{{\mathbb Z}}

\def\M{{\mathcal M}}

\def\End{\mathop{\rm End}\nolimits}
\def\Ker{\mathop{\rm Ker}\nolimits}
\def\Im{\mathop{\rm Im}\nolimits}
\def\Aut{\operatorname{Aut}}

\def\Id{\operatorname{Id}}
\def\tr{\mathop{\rm tr}\nolimits}

\def\Tr{\operatorname{Tr}}

\def\ind{\operatorname{ind}}

\def\1{\mathbf 1}

\newcommand{\ov}[1]{\overline{#1}}

\newcommand{\be}{\begin{enumerate}}
\newcommand{\ee}{\end{enumerate}}
\newcommand{\bq}{\begin{question}}
\newcommand{\eq}{\end{question}}
\newcommand{\bcj}{\begin{conjecture}}
\newcommand{\ecj}{\end{conjecture}}
\newcommand{\bc}{\begin{corollary}}
\newcommand{\ec}{\end{corollary}}
\newcommand{\bl}{\begin{lemma}}
\newcommand{\el}{\end{lemma}}
\newcommand{\btl}{\begin{technicalLemma}}
\newcommand{\etl}{\end{technicalLemma}}
\newcommand{\bp}{\begin{proposition}}
\newcommand{\ep}{\end{proposition}}
\newcommand{\bft}{\begin{fact}}
\newcommand{\eft}{\end{fact}}
\newcommand{\brk}{\begin{remark}}
\newcommand{\erk}{\end{remark}}
\newcommand{\bd}{\begin{Dn}}
\newcommand{\ed}{\end{Dn}}

\def\r{\rho}

\def\w{\omega}

\def\G{{\Gamma}}

\def\C{{\mathbb C}}

\def\M{{\mathcal M}}

\def\End{\mathop{\rm End}\nolimits}
\def\Ker{\mathop{\rm Ker}\nolimits}
\def\Im{\mathop{\rm Im}\nolimits}
\def\Aut{\operatorname{Aut}}

\def\Id{\operatorname{Id}}
\def\tr{\mathop{\rm tr}\nolimits}

\numberwithin{equation}{section}
\newcommand{\Q}{\mathbf Q}

\newcommand{\fix}[1]{\mathop{\textrm{Fix}}(#1)}

\def \crit{\operatorname{crit}}

\def\<{\langle}
\def\>{\rangle}

\newcommand{\de}{\mathrm{d}}

\newcommand{\De}{\mathrm{D}}

\newcommand{\Cc}{\mathcal{C}}

\def\r{\rho}

\def\w{\omega}

\def\G{{\Gamma}}

\renewcommand{\to}{\rightarrow}
\newcommand{\To}{\longrightarrow}
\newcommand{\Mapsto}{\longmapsto}

\newcommand{\inclusion}{\hookrightarrow}
\newcommand{\p}{\partial}

\newcommand{\Of}{\Omega_{\phi}}
\newcommand{\ophi}{\omega_{\phi}}
\DeclareMathOperator{\id}{id}

\DeclareMathOperator{\im}{image}

\DeclareMathOperator{\area}{area}

\DeclareMathOperator{\symp}{Symp}
\DeclareMathOperator{\diff}{Diff}

\DeclareMathOperator{\sign}{sign}
\DeclareMathOperator{\flux}{Flux}
\DeclareMathOperator{\grow}{Growth}

\def\End{\mathop{\rm End}\nolimits}
\def\Ker{\mathop{\rm Ker}\nolimits}
\def\Im{\mathop{\rm Im}\nolimits}
\def\Aut{\operatorname{Aut}}

\def\Id{\operatorname{Id}}
\def\tr{\mathop{\rm tr}\nolimits}

\def\Tr{\operatorname{Tr}}

\def\ind{\operatorname{ind}}

\def \bd {\partial}

\newcommand{\HF}{HF_*}
\newcommand{\CF}{CF_*}

\def\Fix{\operatorname{Fix}}

\def\N{{\mathbb N}}

\begin{document}

\title[New directions in Nielsen-Reidemeister theory]
{New directions in Nielsen-Reidemeister theory}

\author{Alexander Fel'shtyn}
\address{ Institute of Mathematics, University of Szczecin,
ul. Wielkopolska 15, 70-451 Szczecin, Poland
and Department of Mathematics, Boise State University, 1910
University Drive, Boise, Idaho, 83725-155, USA }
\email{felshtyn@diamond.boisestate.edu, felshtyn@mpim-bonn.mpg.de}

\keywords{Nielsen number, Reidemeister number, twisted conjugacy classes,
twisted Burnside-Frobenius theorem, groups with \emph{property } $R_\infty$,
symplectic Floer homology,  Arnold conjecture }
\subjclass{22D10,20E45; 37C25; 43A20; 43A30; 46L; 53D; 37C30; 55M20}

\begin{abstract}
The purpose of this  expository paper is to present new directions
in the classical Nielsen-Reidemeister fixed point theory.
We describe twisted Burnside-Frobenius theorem, groups with  $R_\infty$
\emph{property } and a  connection between Nielsen fixed point theory and
symplectic Floer homology.

\end{abstract}

\maketitle

\tableofcontents

\section{Introduction}

\subsection{ Nielsen- Reidemeister fixed point theory}

 Let  $f: X \rightarrow X $ be a  map
  of a compact  topological space $X$.
 Nielsen-Reidemeister fixed point  theory  suggests a way of counting the fixed points of the map  $f$ in the presence of the fundamental group of the  space $X$. Let $p:\tilde{X}\rightarrow X$ be the
universal covering of $X$ and $\tilde{f}:\tilde{X}\rightarrow \tilde{X}$ a
lifting of $f$, ie. $p\circ\tilde{f}=f\circ p$.  Two liftings $\tilde{f}$ and
$\tilde{f}^\prime$ are called {\it conjugate} if there is a
$\gamma\in\Gamma\cong\pi_1(X)$ such that $\tilde{f}^\prime =
\gamma\circ\tilde{f}\circ\gamma^{-1}$. The subset $p(Fix(\tilde{f}))\subset
Fix(f)$ is called {\it the fixed point class of $f$ determined by the lifting
class $[\tilde{f}]$}. A fixed point class is called $essential$ if its index
is nonzero. The number of lifting classes of $f$ (and hence the number of
fixed point classes, empty or not) is called the {\it Reidemeister Number} of
$f$, denoted $R(f)$. This is a positive integer or infinity. The number of
essential fixed point classes is called the $Nielsen$ $number$ of $f$, denoted
by $N(f)$. The Nielsen number is always finite. $R(f)$ and $N(f)$ are homotopy
type invariants. In the category of compact, connected polyhedra the Nielsen
number of a map is equal to the least number of fixed points of maps with the
same homotopy type as $f$. In Nielsen fixed point   theory the main objects for
investigation are the Nielsen and Reidemeister numbers and their modifications
\cite {j}.

Our definition of a fixed point class is via the universal covering space. It essentially
 says: Two fixed point of $f$ are in the same class iff there is a lifting $\tilde f $ of $f$ having
 fixed points above both of them. There is another way of saying this, which does not use covering space explicitly, hence is  very useful in identifying fixed point classes.Namely,
 two fixed points $x_0$ and $x_1$ of $f$ belong to the same fixed point class iff
 there is a path $c$ from $x_0$ to $x_1$ such that $c \cong f\circ c $ ( homotopy relative endpoints).
  This  can be considered as an equivalent definition of a non-empty fixed point class.
 Every map $f$  has only finitely many non-empty fixed point classes, each a compact
 subset of $X$.
Given a homotopy $H=\{h_t\}: f_0 \cong f_1$, we want to see its influence on fixed point classes
of $f_0$ and $f_1$. A homotopy $\tilde H=\{\tilde h_t\}: \tilde X \to \tilde X$ is called a lifting of the homotopy $H=\{h_t\}$, if $\tilde h_t$ is a lifting of $h_t$ for every $t\in I$.
Given a homotopy $H$ and a lifting $\tilde f_0$ of $f_0$, there is a unique lifting  $\tilde H$
of $H$ such that $ \tilde h_0=\tilde f_0$, hence by unique lifting property of covering spaces
they determine a lifting $ \tilde f_1$ of $f_1$. Thus $H$ gives rise to a one-one correspondence
from liftings of $f_0$ to liftings  of $f_1$.
This correspondence preserves the conjugacy relation.Thus there is a one-to-one
correspondence between lifting classes and fixed point classes of $f_0$ and those of $f_1$.

 Given a selfmap $f:X\to X$ of a compact connected manifold $X$, the nonvanishing of the classical Lefschetz number $L(f)$ guarantees the existence of fixed points. Unfortunately, $L(f)$ yields no information about the size of the set of fixed points of $f$. However, the Nielsen number $N(f)$, a more subtle homotopy invariant, provides a lower bound on the size of this set. For $\dim X\ge 3$, a classical theorem of Wecken   asserts that $N(f)$ is a sharp lower bound on the size of this set, that is, $N(f)$ is the minimal number of fixed points among all maps homotopic to $f$. Thus the computation of $N(f)$ is a central issue in fixed point theory.

Let $G$ be a countable discrete group and $\phi: G\rightarrow G$ an
endomorphism.
Two elements $x,x'\in G$ are said to be
 $\phi$-{\em conjugate} or {\em twisted conjugate,}
iff there exists $g \in G$ with
$
x'=g  x   \phi(g^{-1}).
$
We shall write $\{x\}_\phi$ for the $\phi$-{\em conjugacy} or
{\em twisted conjugacy} class
 of the element $x\in G$.
The number of $\phi$-conjugacy classes is called the {\em Reidemeister number}
of an  endomorphism $\phi$ and is  denoted by $R(\phi)$.
If $\phi$ is the identity map then the $\phi$-conjugacy classes are the usual
conjugacy classes in the group $G$.

Let $f:X\rightarrow X$ be given, and let a
specific lifting $\tilde{f}:\tilde{X}\rightarrow\tilde{X}$ be chosen
as reference.
Let $\Gamma$ be the group of
covering translations of $\tilde{X}$ over $X$.
Then every lifting of $f$ can be written uniquely
as $\alpha\circ \tilde{f}$, with $\alpha\in\Gamma$.
So elements of $\Gamma$ serve as coordinates of
liftings with respect to the reference $\tilde{f}$.
Now for every $\alpha\in\Gamma$ the composition $\tilde{f}\circ\alpha$
is a lifting of $f$ so there is a unique $\alpha^\prime\in\Gamma$
such that $\alpha^\prime\circ\tilde{f}=\tilde{f}\circ\alpha$.
This correspondence $\alpha\rightarrow\alpha^\prime$ is determined by
the reference $\tilde{f}$, and is obviously a homomorphism.
The endomorphism $\tilde{f}_*:\Gamma\rightarrow\Gamma$ determined
by the lifting $\tilde{f}$ of $f$ is defined by
$
  \tilde{f}_*(\alpha)\circ\tilde{f} = \tilde{f}\circ\alpha.
$
It is well known that $\Gamma\cong\pi_1(X)$.
We shall identify $\pi=\pi_1(X,x_0)$ and $\Gamma$ in the usual  way.

We have seen that $\alpha \in \pi$ can be considered as the coordinate of the
lifting $\alpha \circ \tilde f$. We can  tell the conjugacy of two liftings from their coordinates:
$[\alpha \circ\tilde f]=[\alpha^\prime \circ\tilde f] $  iff there is $\gamma  \in \pi$ such that
$\alpha^\prime=\gamma \alpha \tilde{f}_* (\gamma^{-1})$.

So we have the  Reidemeister bijection:
 Lifting classes of $f$ are in 1-1 correspondence with $\tilde{f}_*$-conjugacy classes in group  $\pi$,
 the lifting class $[\alpha\circ\tilde f]$ corresponds to the $\tilde{f}_*$-cojugacy class of $\alpha$.

   By an abuse of language, we  say that the fixed point class $p( \fix{\alpha\circ\tilde f})$,
which is labeled with the lifting class $[\alpha\circ\tilde f]$,corresponds to the $\tilde{f}_*$-conjugacy class of $\alpha$. Thus the $\tilde{f}_*$-conjugacy classes in $\pi$ serve
as coordinates for the fixed point classes of $f$, once a reference lifting $\tilde f$ is chosen.

\subsection{New directions}

The interest in twisted conjugacy relation  for group endomorphism
 has its origins not only 
in  Nielsen-Reidemeister fixed point theory (see, e.g.
 \cite{reid:re,j, FelshB}), but also
in Selberg theory (see, eg. \cite{Shokra,Arthur})  and  Algebraic
 Geometry
(see, e.g. \cite{Groth}).

A current important  problem in this area is to obtain
a twisted analogue of the celebrated Burnside-Frobenius theorem
\cite{FelHill,FelshB,FelTro,FelTroVer,ncrmkwb,cras,polyc,FelTroObzo,IndDipl},
that is, to show the
equality of
the Reidemeister number of $\phi$ and the number of fixed points of the
induced homeomorphism of an appropriate dual object.

If $G$ is a finite group, then the classical Burnside{-Frobenius}
theorem (see, e.g., \cite{serrerepr},
\cite[p.~140]{Kirillov})
says that the number of
classes of irreducible representations is equal to the number of conjugacy
classes of elements of $G$.  Let $\wh G$ be the {\em unitary dual} of $G$,
i.e. the set of equivalence classes of unitary irreducible
representations of $G$.

If $\phi: G\to G$ is an automorphism, it induces a map $\wh\phi:\wh G\to\wh G$,
$\wh\phi (\r)=\r\circ\phi$.
Therefore, by the Burnside{-Frobenius} theorem, if $\phi$ is the identity automorphism
of any finite group $G$, then we have
 $R(\phi)=\#\Fix(\wh\phi)$.

In \cite{FelHill} it was discovered that
this statement remains true for any automorphism $\phi$ of any finite group $G$.
Indeed, if we
consider an automorphism $\phi$ of a finite group $G$, then $R(\phi)$
is equal to the dimension of the space of twisted invariant functions on
this group. Hence, by Peter-Weyl theorem (which asserts the existence of
a two-side equivariant isomorphism
$C^*(G)\cong \bigoplus_{\r\in\wh G} \End(H_\r)$),
$R(\phi)$ is identified
with the sum of dimensions
$d_\r$ of twisted invariant elements of $\End(H_\r)$, where $\r$ runs over
$\wh G$, and the space of representation $\r$ is denoted by $H_\r$.
By the Schur lemma,
$d_\r=1$, if $\r$ is a fixed point of $\wh\phi$, and is zero otherwise. Hence,
$R(\phi)$ coincides with the number of fixed points of $\wh\phi$.

The attempts to generalize this theorem to the case of non-identical
automorphism and of non-finite group
(i.e., to identify the Reidemeister number of $\phi$ and the number
of fixed points of $\wh\phi$ on an appropriate dual object of $G$,
provided that one of these numbers is finite)
were inspired by the dynamical
questions and were the subject  of a series of papers
\cite{FelHill,FelshB,FelTro,FelTroVer,ncrmkwb,FelIndTro,cras,polyc,FelTroObzo,IndDipl}
In the  paper \cite{fts} we studied  the following property for a countable
discrete group $G$ and its automorphism $\phi$: we say that
the group is $\phi$-\emph{conjugacy separable} if its Reidemeister
classes can be distinguished by homomorphisms onto finite groups, and
we say that it is \emph{twisted conjugacy separable} if it is
$\phi$-conjugacy separable for any automorphism $\phi$
with $R(\phi)<\infty$ (\emph{strongly twisted conjugacy separable}, if
we remove this finiteness restriction)
(Definitions \ref{dfn:ficonjsep} and \ref{dfn:twisconjsep}).
This notion was used in \cite{polyc} to prove the twisted Burnside-Frobenius
theorem for polycyclic-by-finite groups with the finite-dimensional part
of the unitary dual $\wh G$ as an appropriate dual object.

In chapter \ref{sec:tbft}, after some preliminary considerations,  we discuss  following 
results:
\begin{enumerate}
    \item  Classes of twisted conjugacy separable groups: Polycyclic-by-finite groups
    are strongly twisted conjugacy separable groups \cite{polyc,fts}.
    \item  Twisted conjugacy separability respects some extensions:
Suppose, there is an extension $H\to G\to G/H$, where the group $H$ is a characteristic
twisted conjugacy separable group; $G/H$ is finitely generated {\rm FC}-group (i.e.,
a group with finite conjugacy classes).
Then $G$ is a twisted conjugacy separable group \cite{polyc,fts}.
    \item  Examples of groups, which are not twisted conjugacy separable:
HNN, Ivanov and Osin groups \cite{polyc}.
\item  The affirmative answer to the twisted Dehn conjugacy problem for
polycyclic-by-finite groups \cite{fts}.
\item   Twisted Burnside-Frobenius theorem  for $\phi$-conjugacy separable groups
in the following formulation:
Let $G$ be an $\phi$-conjugacy separable group .Then $R(\phi)=S_f(\phi)$
if one of these numbers is finite, where $S_f(\phi)$ is the number of fixed points of $\wh\phi:\wh G_f\to \wh G_f$,
$\wh\phi (\r)=\r\circ\phi$, where $\wh G_f$ is the part of the unitary dual
$\wh G$, which is formed by the finite-dimensional representations
\cite{polyc, fts}.

\end{enumerate}

A number of examples
of groups and automorphisms with finite Reidemeister numbers was
obtained and studied in \cite{FelshB,gowon,FelHillWong,FelTroVer,FelIndTro}.
Using the same argument as in \cite{FelTro} one obtains from the
twisted Burnside-Frobenius theorem the following dynamical and number-theoretical
consequence which, together with the twisted Burnside-Frobenius
theorem itself,
is very important for the realization problem of Reidemeister numbers in
topological dynamics and the study of the Reidemeister
zeta-function.
Let $\mu(d)$, $d\in\N$, be the {\em M\"obius function},
i.e.$\mu(d)=0 $ if $d$ is divisible by a square different from one ;
$\mu(d)=(-1)^k $ if
 $d$ is not divisible  by a square different from one , where $k$ denotes the number of
 prime divisors of $d$; $ \mu(1)=1$.

{ Congruences for Reidemeister numbers\cite{polyc}:
{\it
Let $\phi:G$ $\to G$ be an automorphism of a countable discrete
twisted conjugacy separable group $G$
such that all numbers $R(\phi^n)$ are finite.
Then one has for all $n$,}
 $$
 \sum_{d\mid n} \mu(d)\cdot R(\phi^{n/d}) \equiv 0 \mod n.
 $$

These  congruences  were proved previously in a
number of special cases in \cite{FelHill, FelshB,FelTro,FelTroVer,FelIndTro}
and are an analog of remarkable Dold congruences for the  Lefschetz numbers 
of the iterations of a continuous  map\cite{dold}.

One step in the  process to obtain twisted Burnside-Frobenius theorem 
is to describe the class of groups $G$
 for which  $R(\phi)=\infty$ for
any automorphism $\phi:G\to G$.
 We say that  a group $G$ has $R_\infty$
\emph{property }  if all of its automorphisms $\phi$
satisfy  $R(\phi)=\infty$.

The work of discovering which groups  have  $R_\infty$
\emph{property } was begun by Fel'shtyn and Hill in \cite{FelHill}.
 It was later  shown by various authors that the following groups
 have  $R_\infty$ \emph{property }:
 (1) non-elementary Gromov hyperbolic groups \cite{FelPOMI,ll}, (2)
Baumslag-Solitar groups $BS(m,n) = \langle a,b | ba^mb^{-1} = a^n
 \rangle$
except for $BS(1,1)$ \cite{FelGon},
 (3) generalized Baumslag-Solitar groups, that is, finitely generated
 groups
which act on a tree with all edge and vertex stabilizers infinite
 cyclic
\cite{LevittBaums}, (4)
lamplighter groups $\mathbb Z_n \wr \mathbb Z$ if and only if  $2|n$ or
 $3|n$
 \cite{gowon1}, (5)
the solvable generalization $\ga$ of $BS(1,n)$ given by the short exact
 sequence
$1 \rightarrow \mathbb Z[\frac{1}{n}] \rightarrow \ga \rightarrow
 \mathbb Z^k \rightarrow 1,$
 as well as any group quasi-isometric to $\ga$ \cite{TabWong};
such groups  are quasi-isometric to $BS(1,n)$ \cite{TabWong2} ( note
 however that the class of groups
for which $R(\phi)=\infty$ for
any automorphism $\phi$ is not closed under quasi-isometry)  (6)
 saturated weakly branch groups
 (including the Grigorchuk group and the Gupta-Sidki group)
 \cite{FelLeoTro},
(7)The  R. Thompson group F \cite{bfg},
(8)  symplectic groups
  $Sp(2n,\mathbb Z)$,  the mapping class groups  $Mod_{S}$ of a compact  surface $S$ and  the full braid
 groups  $B_n(S)$ on $n$ strings of a compact  surface $S$ in the  cases where
 $S$ is either the compact  disk $D$, or the sphere $S^2$ \cite{dfg}.

The results of the present paper indicate  that the further study
of Reidemeister theory for these groups  should  go
along the lines similar to those  of  the infinite case. On the other
hand, this result reduces  the class of groups for which the
twisted Burnside-Frobenius conjecture
 \cite{FelHill,FelTro,FelTroVer,ncrmkwb,polyc,FelTroObzo,IndDipl}
has yet to be verified.

One is interested in conditions which yield more specific information about the relationship between the Lefshetz number, the Nielsen number and the Reidemeister number.  In the special case of selfmaps of a lens space, Franz  \cite{Franz} showed  that all fixed point classes of such maps have the same fixed point index.  From this it follows that one of two situations occurs, namely
1) $L(f)=0$ implies $N(f)=0$, in which case $f$ is deformable to a fixed point free map, and
2) $L(f)\ne 0$ implies $N(f)=\#Coker (1-f_{*_1})$ where $f_{*_1}$ is the induced homomorphism on the first integral homology.
When the Lefschetz number $L(f)$ is nonzero, the cardinality of $Coker (1-f_{*_1})$ is exactly equal to the Reidemeister number $R(f)$.

The evaluation subgroup $J(X)$, also known as the first Gottlieb group $G_1(X)$, of the fundamental group $\pi$ is the image of the homomorphism $ev_*:\pi_1(X^X,1_X) \to \pi_1(X,x_0)$ induced by the evaluation map at a point $x_0\in X$. If $J(X)$ coincides with $\pi$, the analogous results are obtained as well. Namely, $L(f)=0$ implies that $N(f) = 0$, and $L(f)\ne 0$ implies that  $N(f)=R(f)$ for all maps $f$. This general type of result was first proved by B. Jiang  (see \cite{j} for more information on Nielsen fixed point theory) for a class of spaces now known as Jiang spaces.
Jiang-type results hold for all selfmaps of a large class of spaces including simply-connected spaces, generalized lens spaces, $H$-spaces, topological groups, orientable coset spaces of compact connected Lie groups, nilmanifolds, certain $\mathcal C$-nilpotent spaces where $\mathcal C$ denotes the class of finite groups, certain solvmanifolds and infra-homogeneous spaces  \cite{Wo2}.  When $\dim X\ge 3$ and $N(f)=0$, the selfmap $f$ is deformable to be fixed point free by Wecken's theorem. For the equality $N(f)=R(f)$ to hold, one must first determine the finiteness of the Reidemeister number $R(f)$.

Unless $N(f)=0$ for all homeomorphisms $f$, property $R_\infty$ eliminates the possibility of a Jiang-type result, which would require the finiteness of the Reidemeister number.
  That is, if a group $G$ has property $R_\infty$, a compact connected manifold with $G$ as fundamental group will never satisfy a Jiang-type result.

In particular the fundamental group of surface of genus greater than 1 has
property $R_\infty$.  So,   we need a new ideas and tools in this classical case. Recently a  connection between symplectic Floer homology and Nielsen fixed
point theory was discovered \cite{ff, G}. The author came to the  idea that Nielsen numbers are  connected with  Floer homology of surface diffeomorphisms  at the Autumn 2000, after conversations with Joel Robbin and Dan Burghelea.

In the chapter \ref{sec:floer} we discuss  the connection between  symplectic Floer homology group and Nielsen fixed point theory.
We also describe symplectic zeta functions and an  asymptotic invariant of monotone symplectomorphism. Generalised Arnold conjecture is formulated.

A part of this article  grew out of the joint papers of the author with E.V. Troitsky  and  D. L. Gon\c{c}alves   written in 2001- 2007. Their collaboration is gratefully appreciated.
 
The author is  grateful to
 B. Bowditch,  F. Dahmani, Ya. Eliashberg,  M. Gromov,  V. Guirardel,   Wu-Chung Hsiang,  M. Kapovich, G. Levitt,   J.D. McCarthy,  K. Ono,  M. Sapir,  Y. Rudyak,   S. Sidki,  V. Turaev and  A. Vershik   for stimulating discussions and comments.

Parts of this article were written while the author was was visiting the
Max-Planck-Institute f\"ur Mathematik,
Bonn in 2001-2006. 
The author would like to thank the Max-Planck-Institute f\"ur Mathematik,
Bonn for kind hospitality and support.

\section{Twisted Burnside-Frobenius theorem}\label{sec:tbft}

\subsection{Preliminary considerations}\label{sec:prelim}

 Following construction relates $\phi$-conjugacy classes and
some conjugacy classes of another group.
Consider the action of $\Z$ on $G$, i.e. a homomorphism $\Z\to\Aut(G)$,
$n\mapsto\phi^n$. Let $\G$ be a corresponding semi-direct product
$\G=G\rtimes\Z$:
\begin{equation}\label{eq:defgamma}
    \G:=< G,t\: |\: tgt^{-1}=\phi(g) >
\end{equation}
in terms of generators and relations, where $t$ is a generator of $\Z$.
The group $G$ is a normal subgroup
of $\G$. As a set, $\G$ has the form
\begin{equation}\label{eq:razlgamm}
    \G=\sqcup_{n\in\Z}G\cdot t^n,
\end{equation}
where $G\cdot t^n$ is the coset by $G$ containing $t^n$.

\begin{rk}
Any usual conjugacy class of $\G$ is contained in some $G\cdot t^n$.
Indeed, $g g't^n g^{-1}=gg'\phi^n(g^{-1}) t^n$ and $t g't^n t^{-1}=\phi(g')t^n$.
\end{rk}

\begin{lemma}
\label{lem:HillRanbij}
Two elements $x,y$ of $G$ are $\phi$-conjugate iff $xt$ and $yt$ are conjugate
in the usual sense in $\G$. Therefore $g\mapsto g\cdot t$ is a bijection
from the set of $\phi$-conjugacy classes of $G$ onto the set of conjugacy
classes of $\G$ contained in $G\cdot t$.
\end{lemma}

\begin{proof}
If $x$ and $y$ are $\phi$-conjugate then there is a $g\in G$ such that
$gx=y\phi(g)$. This implies $gx=ytgt^{-1}$ and therefore $g(xt)=(yt)g$
so $xt$ and $yt$ are conjugate in the usual sense in $\G$. Conversely,
suppose $xt$ and $yt$ are conjugate in $\G$. Then there is a $gt^n\in\G$
with $gt^nxt=ytgt^n$. From the relation $txt^{-1}=\phi(x)$ we obtain
$g\phi^n(x)t^{n+1}=y\phi(g)t^{n+1}$ and therefore $g\phi^n(x)=y\phi(g)$.
Hence, $y$ and $\phi^n(x)$ are $\phi$-conjugate. Thus,
$y$ and $x$ are $\phi$-conjugate, because $x$ and $\phi(x)$ are always
$\phi$-conjugate: $\phi(x)=x^{-1} x \phi(x)$.
\end{proof}

\subsection{Twisted conjugacy separability}\label{sec:separat}
We would like to give a generalization of the following well known notion.

\begin{dfn}
A group $G$ is \emph{conjugacy separable} if any pair $g$, $h$ of
   non-conjugate elements of $G$ are non-conjugate in some finite
   quotient of $G$.
\end{dfn}

It was proved that polycyclic-by-finite groups
are conjugacy separated (\cite{Remes,Form}, see also \cite[Ch.~4]{DSegalPoly}).

We can introduce the following notion, which coincides with the previous
definition in the case $\phi=\Id$.

\begin{dfn}\label{dfn:ficonjsep}\cite{fts}
A group $G$ is \emph{$\phi$-conjugacy separable} with respect to
an automorphism $\phi:G\to G$ if any pair $g$, $h$ of
   non-$\phi$-conjugate elements of $G$ are non-$\ov\phi$-conjugate in some finite
   quotient of $G$ respecting $\phi$.
\end{dfn}

This notion is closely related to the notion ${\rm RP}(\phi)$
introduced in \cite{polyc}.

\begin{dfn}\label{dfn:Rperiodi}\cite{polyc}
We say that a group $G$ has the property   {\rm  RP} if for
any automorphism $\phi$ with $R(\phi)<\infty$ the
characteristic functions $f$ of {\rm  Reidemeister} classes (hence
all $\phi$-central functions) are   {\rm  periodic} in the
following sense.

There exists a finite group $K$,
its automorphism $\phi_K$, and epimorphism $F:G\to K$ such that
\begin{enumerate}
    \item The diagram
    $$
    \xymatrix{
G\ar[r]^\phi\ar[d]_F& G\ar[d]^F\\
K\ar[r]^{\phi_K}& K
    }
    $$
    commutes.
    \item $f=F^*f_K$, where $f_K$ is a characteristic function
    of a subset of $K$.
\end{enumerate}

If this property holds for a concrete automorphism $\phi$, we
will denote this by RP($\phi$).
\end{dfn}

One gets immediately the following statement.

\begin{theorem}\label{teo:phiconjRP}\cite{polyc}
Suppose, $R(\phi)<\infty$. Then $G$ is $\phi$-conjugacy separable if and
only if $G$ is {\rm RP}$(\phi)$.
\end{theorem}

\begin{proof}
Indeed, let $F_{ij}:G\to K_{ij}$ distinguish $i$th and $j$th $\phi$-conjugacy
classes, where $K_{ij}$ are finite groups, $i,j=1,\dots,R(\phi)$. Let
$F:G\to \oplus_{i,j}K_{ij}$, $F(g)=\sum_{i,j}F_{ij}(g)$, be the diagonal
mapping and $K$ its image. Then $F:G\to K$ gives RP$(\phi)$.

The opposite implication is evident.
\end{proof}

\begin{dfn}\label{dfn:twisconjsep}\cite{fts}
A group $G$ is \emph{twisted conjugacy separable} if it is $\phi$-conjugacy
separable for any $\phi$ with $R(\phi)<\infty$.

A group $G$ is \emph{strongly twisted conjugacy separable} if it is $\phi$-conjugacy
separable for any $\phi$.
\end{dfn}

From Theorem \ref{teo:phiconjRP} one immediately obtains

\begin{corollary}\label{cor:twistconjRP}\cite{fts}
A group $G$ is twisted conjugacy separable if and only if it is {\rm RP}.
\end{corollary}

\begin{theorem}\label{teo:razdelgamg}\cite{polyc}
Let $F:\G\to K$ be a morphism onto a finite group $K$ which separates
two conjugacy classes of $\G$ in $G\cdot t$. Then the restriction
$F_G:=F|_G:G\to \Im(F|_G)$
separates the corresponding $($by the bijection from
Lemma {\rm \ref{lem:HillRanbij})} $\phi$-conjugacy classes in $G$.
\end{theorem}

\begin{proof}
First of all let us remark that $\Ker(F_G)$ is $\phi$-invariant.
Indeed, suppose $F_G(g)=F(g)=e$. Then
$$
F_G(\phi(g))=F(\phi(g))=F(tgt^{-1})=F(t)F(t)^{-1}=e
$$
(the kernel of $F$ is a normal subgroup).

Let $gt$ and $\til gt$ be some representatives of the mentioned
conjugacy classes. Then
$$
F((ht^n)gt (ht^n)^{-1})\ne F(\til gt),\qquad\forall h\in G,\: n\in \Z,
$$
$$
F(ht^ngt )\ne F(\til gt ht^n),\qquad\forall h\in G,\: n\in \Z,
$$
$$
F(h\phi^n(g)t^{n+1} )\ne F(\til g \phi(h)t^{n+1}),\qquad\forall h\in G,\: n\in \Z,
$$
$$
F(h\phi^n(g))\ne F(\til g \phi(h)),\qquad\forall h\in G,\: n\in \Z,
$$
in particular, $F(hg\phi(h^{-1}))\ne F(\til g )$ $\forall h\in G$.
\end{proof}

\begin{theorem}\label{teo:conjsepandRP}\cite{fts}
Let some class of conjugacy separable groups be closed under
taking semidirect products by $\Z$. Then this class consists of
strongly twisted conjugacy separable groups.
\end{theorem}

\begin{proof}
This follows immediately from Theorem \ref{teo:razdelgamg} and
Theorem \ref{teo:phiconjRP}.
\end{proof}

\subsubsection{First examples: polycyclic-by-finite groups}\label{sec:almostpolsep}

As an application we obtain another proof of the main theorem for
polycyclic-by-finite groups.

Let $G'=[G,G]$ be the \emph{commutator subgroup} or
\emph{derived group} of $G$, i.e. the subgroup generated by
commutators. $G'$ is invariant under any homomorphism, in
particular it is normal. It is the smallest normal subgroup of $G$
with an abelian factor group. Denoting $G^{(0)}:=G$, $G^{(1)}:=G'$,
$G^{(n)}:=(G^{(n-1)})'$, $n\ge 2$, one obtains \emph{derived series}
of $G$:
\begin{equation}\label{eq:derivedseries}
    G=G^{(0)}\supset G'\supset G^{(2)}\supset \dots\supset G^{(n)}\supset
\dots
\end{equation}
If $G^{(n)}=e$ for some $n$, i.e. the series (\ref{eq:derivedseries})
    stabilizes by trivial group,
    the group $G$ is \emph{solvable};

\begin{dfn}\label{dfn:polycgroup}
A solvable group with derived series with cyclic factors is called
\emph{polycyclic group}.
\end{dfn}

\begin{theorem}\label{teo:almostpolaresep}\cite{polyc}
Any polycyclic-by-finite group is a strongly twisted conjugacy separable group.
\end{theorem}

\begin{proof}
The class of polycyclic-by-finite groups is closed under
taking semidirect products by $\Z$. Indeed, let $G$ be an
 polycyclic-by-finite group. Then there exists a characteristic
 (polycyclic) subgroup $P$ of finite index in $G$. Hence,
 $P\rtimes \Z$ is a polycyclic normal group of $G\rtimes \Z$
 of the same finite index.

Polycyclic-by-finite groups are
conjugacy separable (\cite{Remes,Form}, see also \cite[Ch.~4]{DSegalPoly}).
It remains to apply Theorem \ref{teo:conjsepandRP}.
\end{proof}

\subsection{Twisted conjugacy separability and extensions}
It is known that conjugacy separability does not respect extensions.
For twisted conjugacy separable groups the situation is much better under
some finiteness conditions. More precisely one has the following statement.

\begin{theorem}\label{teo:phiconjandexten}\cite{polyc}
Suppose, there exists a commutative diagram
\begin{equation}\label{eq:extens}
 \xymatrix{
0\ar[r]&
H \ar[r]^i \ar[d]_{\phi'}&  G\ar[r]^p \ar[d]^{\phi} & G/H \ar[d]^{\ov{\phi}}
\ar[r]&0\\
0\ar[r]&H\ar[r]^i & G\ar[r]^p &G/H\ar[r]& 0,}
\end{equation}
where $H$ is a normal subgroup of a finitely generated group $G$.
Suppose, $R(\phi)<\infty$, $G/H$ is a {\rm FC} group, i.e., all conjugacy
classes are finite, and $H$ is a $\phi'$-conjugacy separable group. Then
$G$ is a $\phi$-conjugacy separable group.
\end{theorem}

Another variant of finiteness is $|G/H|<\infty$ (without the property $R(\phi)<\infty$).

\subsection{ Twisted Burnside-Frobenius theorem for $\phi$-conjugacy 
separable groups }\label{sec:tbftsep}

\begin{dfn}
Denote by $\wh G_f$ the subset of the unitary dual $\wh G$ related to
finite-di\-men\-si\-on\-al representations.
\end{dfn}

\begin{theorem}\label{teo:twburnsforRP}\cite{polyc}
Let $G$ be an $\phi$-conjugacy separable group .Then $R(\phi)=S_f(\phi)$
if one of these numbers is finite.
\end{theorem}

\begin{proof}
The coefficients of finite-dimensional non-equivalent irreducible
representations of $G$ are linear independent by Frobenius-Schur theorem
(see \cite[(27.13)]{CurtisReiner}). Moreover, the coefficients of non-equivalent unitary
finite-dimensional irreducible representations are orthogonal to each other
as functions on the universal compact
group associated with the initial group \cite[16.1.3]{DixmierEng} by the
Peter-Weyl theorem. Hence, their linear combinations are orthogonal to each other
as well.

It is sufficient to verify the following three statements:

1) If $R(\phi)<\infty$, then each $\phi$-class function is a finite
linear combination of twisted-invariant functionals being coefficients
of points of $\Fix\wh\phi_f$.

2) If $\rho\in\Fix\wh\phi_f$, there exists one and only one
(up to scaling)
twisted invariant functional on $\rho(C^*(G))$ (this is a finite
full matrix algebra).

3) For different $\rho$ the corresponding $\phi$-class functions are
linearly independent. This follows from the remark at the beginning of
the proof.
Let us remark that the property RP implies in particular that
$\phi$-central functions (for $\phi$ with $R(\phi)<\infty$) are
functionals on $C^*(G)$, not only $L^1(G)$, i.e. are in the
Fourier-Stieltijes algebra $B(G)$.

The statement 1) follows from the RP property. Indeed,
this $\phi$-class function $f$ is a linear combination of functionals
coming from some finite collection $\{\r_i\}$ of elements
of $\wh G_f$ (these
representations $\r_1,\dots,\r_s$ are in fact representations of
the form $\pi_i\circ F$, where $\pi_i$ are irreducible representations
of the finite group $K$ and $F:G\to K$, as in the definition of RP).
So,
$$
f=\sum_{i=1}^s f_i\circ \r_i,\quad \r_i:G\to \End (V_i),
\quad f_i:\End (V_i)\to \C,  \r_i\ne \r_j,\: (i\ne j).
$$
For any $g, \til g \in G$ one has
$$
\sum_{i=1}^s f_i (\r_i(\til g ))=f(\til g )=
f(g\til g \phi(g^{-1}))=\sum_{i=1}^s f_i (\r_i(g\til g \phi(g^{-1}))).
$$
By the observation at the beginning of the proof concerning
linear independence,
$$
f_i (\r_i(\til g ))=f_i (\r_i(g\til g \phi(g^{-1}))).\qquad i=1,\dots,s,
$$
i.e. $f_i$ are twisted-invariant.  For any $\r\in \wh G_f$,
$\r:G\to \End (V)$, any
functional $\w: \End (V)\to \C$
has the form $a\mapsto \Tr(ba)$ for some fixed $b\in \End (V)$.
Twisted invariance implies twisted invariance of $b$ (evident details can
be found in \cite[Sect. 3]{FelTro}). Hence, $b$ is intertwining between
$\r$ and $\r\circ\phi$ and $\r\in \Fix(\wh\phi_f)$. The uniqueness of
intertwining operator (up to scaling) implies 2).

It remains to prove that $R(\phi)<\infty$ if $S_f(\phi)<\infty$.
By the definition of a $\phi$-conjugacy separable group
the Reidemeister classes of $\phi$ can be separated by maps to finite groups.
Hence, taking representations of these finite groups and applying
the twisted Burnside-Frobenius theorem to these groups we obtain that
for any pair of Reidemeister classes there exists a function being a
coefficient of a finite-dimensional unitary representation, which
distinguish these classes. Hence,
if $R(\phi)=\infty$, then there are infinitely
many linearly independent twisted invariant functions being
coefficients of finite dimensional
representations. But there are as many such functionals, as $S_f(\phi)$.

\end{proof}

\begin{corollary}\cite{polyc}
Let $G$ be almost polycyclic group and $\phi$  its automorphism.
Then $R(\phi)=S_f(\phi)$
if one of these numbers is finite.
\end{corollary}

\subsection{Examples and counterexamples}\label{sec:osin}

Some of examples of groups, for which the twisted  Burnside-Frobenius
theorem in the above formulation is true, out of the class of
polycyclic-by-finite groups were obtained by F.~Indukaev \cite{IndDipl}. Namely,
it is proved that wreath products $A\wr\Z$ are RP groups,
where $A$ is a finitely
generated abelian group (these groups are residually finite).

Now let us present some counterexamples to the twisted Burnside-Frobenius
theorem in the above formulation for some discrete groups with extreme properties.
Suppose, an infinite discrete group
$G$ has a finite number of conjugacy classes.
Such examples can be found in \cite{serrtrees} (HNN-group),
\cite[p.~471]{olsh} (Ivanov group), and \cite{Osin} (Osin group).
Then evidently, the characteristic function of the unity element is not
almost-periodic and the argument above is not valid. Moreover, let us
show, that these groups give rise counterexamples to the above theorem.

In particular, they are not twisted conjugacy separable. Evidently,
they are not conjugacy separable, because they are not residually finite.

\begin{ex}\label{ex:osingroup}\cite{polyc}
For the Osin group  Reidemeister number $R(\Id)=2$,
while there is only trivial (1-dimensional) finite-dimensional
representation.
Indeed, Osin group is
an infinite finitely generated group $G$ with exactly two conjugacy classes.
All nontrivial elements of this group $G$ are conjugate. So, the group $G$
is simple, i.e. $G$ has no nontrivial normal subgroup.
This implies that group $G$ is not residually finite
(by definition of residually finite group). Hence,
it is not linear (by Mal'cev theorem \cite{malcev}, \cite[15.1.6]{Robinson})
and has no finite-dimensional irreducible unitary
representations with trivial kernel. Hence, by simplicity of $G$, it has no
finite-dimensional
irreducible unitary representation with nontrivial kernel, except of the
trivial one.

Let us remark that Osin group is non-amenable, contains the free
group in two generators $F_2$,
and has exponential growth.
\end{ex}

\begin{ex}\label{ex;ivanovgroup}\cite{polyc}
For large enough prime numbers $p$,
the first examples of finitely generated infinite periodic groups
with exactly $p$ conjugacy classes were constructed
by Ivanov as limits of hyperbolic groups (although hyperbolicity was not
used explicitly) (see \cite[Theorem 41.2]{olsh}).
Ivanov group $G$ is infinite periodic
2-generator  group, in contrast to the Osin group, which is torsion free.
The Ivanov group $G$ is also a simple group.
The proof (kindly explained to us by M. Sapir) is the following.
Denote by $a$ and $b$ the generators of $G$ described in
\cite[Theorem 41.2]{olsh}.
In the proof of Theorem 41.2 on  \cite{olsh}
it was shown that each of elements of $G$
is conjugate in $G$ to a power of generator $a$ of order $s$.
Let us consider any normal subgroup $N$ of $G$.
Suppose
$\gamma \in N$. Then $\gamma=g a^sg^{-1}$ for some $g\in G$ and some $s$.
Hence,
$a^s=g^{-1} \gamma g \in N$ and from  periodicity of $a$, it follows that also
$ a\in N$
as well as $ a^k \in N$  for any $k$, because $p$ is prime.
Then any element $h$ of $G$ also belongs to $N$
being of the form $h=\til h a^k (\til h)^{-1}$, for  some $k$, i.e., $N=G$.
Thus, the group $G$ is simple. The discussion can be completed
in the same way as in the case of Osin group.
\end{ex}

\begin{ex}
In paper \cite{HNN}, Theorem III and its corollary,
G.~Higman, B.~H.~Neumann, and H.~Neumann
proved that any locally infinite countable group $G$
can be embedded into a countable group $G^*$ in which all
elements except the unit element are conjugate to each other
(see also \cite{serrtrees}).
The discussion above related Osin group remains valid for $G^*$
groups.
\end{ex}

\subsection{Twisted Dehn conjugacy problem }\label{sec:dehn}

The subject is closely related to some decision problem. Recall that
M.~Dehn in 1912 \cite{Dehn} (see \cite[Ch. 1, \S 2; Ch. 2, \S 1]{LyndSchupp}))
has formulated in particular

\smallskip\noindent
\textbf{Conjugacy problem:} Does there exists an algorithm to determine
whether an arbitrary pair of group words $U$, $V$ in the generators of $G$ define
conjugate elements of $G$?

\smallskip
 Following question was posed by G.~Makanin \cite[Question 10.26(a)]{Kourovka}:

\smallskip\noindent
\textbf{Question:} Does there exists an algorithm to determine
whether for an arbitrary pair of group words $U$ and $V$
of a free group $G$ and an arbitrary automorphism $\phi$ of $G$
the equation $\phi(X)U=VX$ solvable in $G$?

In  \cite{BogMarVen}  the  affirmative answer  to the Makanin's question
is obtained.

\smallskip
In  \cite{BardBokutVesnin} the following problem, which generalizes the two
above problems, was posed:

\smallskip\noindent
\textbf{Twisted conjugacy problem:} Does there exists an algorithm to determine
whether for an arbitrary pair of group words $U$ and $V$ in the generators of $G$
the equality $\phi(X)U=VX$ holds for some $W\in G$ and $\phi\in H$, where $H$ is
a fixed subset of $\Aut(G)$?

\smallskip

We will discuss the twisted conjugacy problem for $H=\{\phi\}$.

\begin{theorem}\label{teo:twistconjprobforpolbfin}\cite{fts}
The twisted conjugacy problem has the affirmative answer for $G$ being
polycyclic-by-finite group and $H$ be equal to a unique automorphism $\phi$.
\end{theorem}

\begin{proof}
It follows immediately from Theorem \ref{teo:almostpolaresep} by the same
argument as in the paper of
Mal'cev \cite{MalcevIvanovo} (see also \cite{Mostowski}, where the propperty
of conjugacy separability was first formulated)
for the (non-twisted) conjugacy problem.
\end{proof}
In fact we have proved the following statement.
\begin{theorem}\cite{fts}
If $G$ is strongly twisted conjugacy separable then the twisted
Dehn conjugacy problem is solvable for any  automorphism of $G$.
\end{theorem}
\begin{corollary}
 The twisted conjugacy problem has the affirmative answer for $G$ being  wreath products $A\wr\Z$ and any  automorphism of $G$,
where $A$ is a finitely
generated abelian group.

\end{corollary}

\begin{proof}
It is proved in  \cite {IndDipl}  that $G$ is strongly twisted conjugacy separable.
\end{proof}

Also one can study some more particular cases of this problem.
In particular, one has
\begin{theorem}\cite{fts}
Let $G$ be a $\phi$-conjugacy separable group. Then the twisted
Dehn conjugacy problem is solvable for $\phi$.
\end{theorem}

From Corollary 3.4 and Proposition 3.5 in \cite{Martino} one can
obtain
\begin{theorem}\cite{fts}
Suppose $G$ is the fundamental group of a closed hyperbolic surface
and $\phi:G\to G$ is virtually inner. Then the twisted
Dehn conjugacy problem is solvable for $\phi$.
\end{theorem}

Recently, in \cite{BogMarVen1} the twisted Dehn conjugacy problem
was solved for virtually surface groups and an example of a finitely
presented group with solvable conjugacy problem but unsolvable
twisted conjugacy problem was given.

\subsection{Some questions}\label{sec:quest}\cite{fts}
It is evident, that any conjugacy separable group is residually
finite (because the unity element is an entire conjugacy class).
This argument does not work for general Reidemeister classes.
In this relation we have  formulated in \cite{fts} several open  questions:

\smallskip
\textbf{Question 1:} Does the $\phi$-conjugacy separability imply
residually finiteness?

\smallskip
\textbf{Question 2:} Does the $\phi$-conjugacy separability imply
residually finiteness, provided $R(\phi)<\infty$?

\smallskip
\textbf{Question 3:} Does the twisted conjugacy separability imply
residually finiteness, provided the existence of $\phi$ with
$R(\phi)<\infty$?

\smallskip
\textbf{Question 4:} Let $G$ be a residually finite group and
$\phi$ its automorphism with $R(\phi)<\infty$. Is $G$
$\phi$-conjugacy separable ?

\smallskip
The affirmative answer to the last question implies twisted
Burnside-Frobenius theorem for $\phi$.

\section{Groups with \emph{property } $R_\infty$}

Consider a group extension respecting homomorphism $\phi$:
\begin{equation}\label{eq:extens}
 \xymatrix{
0\ar[r]&
H \ar[r]^i \ar[d]_{\phi'}&  G\ar[r]^p \ar[d]^{\phi} & G/H \ar[d]^{\ov{\phi}}
\ar[r]&0\\
0\ar[r]&H\ar[r]^i & G\ar[r]^p &G/H\ar[r]& 0,}
\end{equation}
where $H$ is a normal subgroup of $G$.
First, notice that the Reidemeister classes of $\phi$ in $G$
are mapped epimorphically onto classes of $\ov\phi$ in $G/H$. Indeed,
\begin{equation}\label{eq:epiofclassforexs}
p(\til g) p(g) \ov\phi(p(\til g^{-1}))= p (\til g g \phi(\til g^{-1}).
\end{equation}
Suppose that the Reidemeister number  $R(\ov\phi)$ is infinite, the previous remark then implies that the Reidemeister number
$R(\phi)$ is infinite.

Let $G$ be a group, and let  $\varphi$ be  an automorphism of $G$
of order $m$. $G_\varphi$ be the group $G\rtimes_\varphi \Z_m = \langle
 G,t | \, \forall g\in G,
tgt^{-1} = \varphi(g),\, t^m=1 \rangle $.

The following lemma was proven by Delzant.

\begin{lemma}\label{del} \cite[Lemma 3.4]{ll}
If $K$  is a normal subgroup of a group $\Gamma$
acting
non-elementarily on a hyperbolic space, and if $\Gamma/K$ is Abelian,
then
any coset of $K$ contains infinitely many conjugacy classes.
\end{lemma}

\begin{theorem}\label{dam}\cite{dfg}
 If $G_\varphi$ has a non-elementary action by isometries on a
Gromov-hyperbolic length space, then $G$ has infinitely many
$\varphi$-twisted conjugacy classes.
\end{theorem}

\begin{proof}
By elementary action, we mean an action consisting of elliptic
 elements,
or
with a global fixed point, or a global fixed pair, in the boundary of
the
hyperbolic space.
The  statement of the theorem follows  immediately from Lemma  \ref{lem:HillRanbij} and Delzant Lemma \ref{del}
\end{proof}

\begin{theorem}

 The following groups  have  $R_\infty$ \emph{property }:
 
(1) non-elementary Gromov hyperbolic groups \cite{FelPOMI,ll}

(2) non-elementary  relatively hyperbolic groups,

(3) the mapping class groups  $Mod_{S}$ of a compact  surface $S$(with a  few exceptions) \cite{dfg},

(4) the full braid groups  $B_n(S^2)$ on $n$ strings of  the sphere $S^2$ \cite{dfg},

\end{theorem}

\begin{proof}

(1)-(2): 
 Theorem \ref{dam} applies if $G$ is a Gromov-hyperbolic group 
 or relatively hyperbolic group and
if $\varphi$ has finite order in $Out(G)$. In fact, in this case,
$G_\varphi$ contains $G$ as a subgroup of finite index, thus is
quasi-isometric to $G$, and by quasi-isometry invariance, it is itself
 a Gromov-hyperbolic or relatively hyperbolic group.
Now let  assume that  a automorphism of a hyperbolic or relatively hyperbolic group has  infinite order in $Out(G)$. We describe main steps of the proof in this case ( see \cite{FelPOMI,ll}  for the  details).
By  \cite{p} and \cite{bsz} $\Phi$ preserves some $R$-tree $T$
with nontrivial minimal small action of $G$( recall that an action of $G$ is small if all ars stabilisers are virtually cyclic; the action of $G$ on $T$
is always irreducible( no global fixed point, no invariant line, no invariant end)). This means that there is an $R$-tree  $T$ equipped with
an isometric action of $G$ whose length function satisfies
$l\cdot \Phi=\lambda l$ for some $\lambda \geq 1$.\\
Step 1. Suppose $ \lambda =1$.  Then the Reidemeister number $R(\phi)$  is infinite .\\
Step 2. Suppose $\lambda >1$. Assume that arc stabilisers are finite, and
there exists $N_0\in N$ such that, for every $Q\in T$, the action
of $Stab Q$ on $\pi_o(T-{Q})$ has at most $N_0$ orbits.
Then  the Reidemeister number $R(\phi)$  is infinite.\\
Step 3. If $\lambda >1$, then $T$ has finite arc stabilisers.
If $\lambda >1$ then from work of Bestwina-Feighn \cite{bf} it follows
that there exists $N_0\in N$ such that, for every $ Q\in T$, the action
of $Stab Q$ on $\pi_o(T-(Q))$ has at most $N_0$ orbits.

(3)  Now let  $S$ be an oriented, compact surface of  genus $g$ and with
  $p$
boundary components, where  $3g+p-4 >0$. It is easy to see that the
 mapping Class Group
$Mod_S$ is a normal subgroup of  the full mapping class group $
 Mod_S^*$, of  index  $2$.
The graph of curves of $S$,  denoted $\mathcal{G}(S)$,  is the
graph
whose vertices are the simple curves of $S$ modulo isotopy. Two
vertices (that is two isotopy classes of simple curves) are linked by
an
edge in this graph if they can be realized by disjoint curves. Both
$Mod_S$ and  $ Mod_S^*$ act on $\mathcal{G}(S)$ in a
non-elementary way.
Now we  use the non-elementary
result of Masur and Minsky
\cite{MM} (see also Bowditch \cite{B}) that the complex of curves of an
oriented surface (with genus $g$ and $p$ boundary components, and
$3g+p-4
>0$) is Gromov-hyperbolic space.

Thus Theorem \ref{dam}
is applicable  for $Mod_S$ and
 for $\varphi_1$  the automorphism induced by reversing the
orientation of $S$, since in this  case,
$({Mod_S})_{\varphi_1}={{Mod_S}\rtimes}_{\varphi_1} \Z_2 \simeq
 Mod_S^*$.
For $Mod_S$ and
$\varphi_0= Id$, we have $R(\varphi_0= Id)=\infty$ because the group
  $Mod_S$
has infinite an number of usual conjugacy classes.
Finally,
 $Out(Mod_S) \simeq \{\overline{\varphi_0}, \overline{\varphi_1}
\}$ (see  \cite{iva}),
which   ensures   that $Mod_S$ has the  $R_\infty$
property if  $S$ is  an orientable,  compact surface of  genus $g$ with
  $p$
boundary components, where  $3g+p-4 >0$. The only cases not covered by
 this inequality are:
i) S is the torus with at most one hole  and  ii) S is  the sphere with
 at most 4
 holes. The case of the torus with at most one hole  follows from
the  section 3 in \cite{dfg}. In the case of the  sphere with at most 4 holes, this
  follows directly
 from the knowledge of Out(ModS)  and the cardinality of the mapping
 class group.

(4) Let $\phi: B_n(S^2) \to B_n(S^2)$ be an automorphism.
 Since the center of   $B_n(S^2)$  is a characteristic subgroup, $\phi$
 induces a homomorphism of the short  exact sequence

$$1\to \mathbb Z_2 \to  B_{n}(S^2)\to Mod_{S^2_n} \to 1,$$
\noindent
where $S^2_r=S^2-r$ open disks. This   short exact sequence was   obtained from the sequence
 in \cite{Bi}. The result above  implies 
 that  the group $Mod_{S^2_n}$ for  $n>3$,  has
the $R_{\infty}$ property. Then the remark  about extensions
 implies that  the group  $B_{n}(S^2)$ also has this property.
 For $n\leq 3$ the groups $B_n(S^2)$ are finite so they do not have the
 $R_{\infty}$ property.

\end{proof}

Relatively hyperbolic groups were introduced by Gromov\cite{gromov} and since
then various characterizations of relatively hyperbolic groups have been obtained.
Examples of relatively hyperbolic groups are:
\begin{itemize}
\item the free products of finitely many finitely generated
groups are hyperbolic relative to the factors.

\item
geometrically finite isometry groups of Hadamard
manifolds of negatively pinched sectional curvature
 are hyperbolic relative to
the maximal parabolic subgroups. This includes complete
finite volume manifolds of negatively pinched 
sectional curvature.

\item
The amalgamation of relatively hyperbolic groups over
parabolic subgroups
is relatively hyperbolic, when the parabolic subgroup
is maximal in at least one of the factors \cite{Dah-comb, Osi-comb}
\item
 $\mathrm{CAT}(0)$-groups with isolated flats
are hyperbolic relative to the flat stabilizers. Examples of
$\mathrm{CAT}(0)$-groups with isolated flats are listed 
in~\cite{Hru-Kle}.
\item Sela's limit groups  are hyperbolic relative to non-cyclic
maximal abelian subgroups \cite{Dah-comb}
\end{itemize}

\subsection{Asymptotic expansions} \label{asympt}

Suppose  we know that the number of the twisted conjugacy classes of a automorphism $\phi$ of a  group $G$ is infinite. The next natural step will be to write an asymptotic
for the number of twisted conjugacy classes with a  norm smaller then $x$.
First of all  we need to define a norm of a twisted conjugacy class. In present  section
we realise this approach for pseudo-Anosov homeomorphism of a compact surface.

We assume $X$ to be a compact surface of negative Euler characteristic and
 $f:X\rightarrow X$ is a pseudo-Anosov homeomorphism, i.e. there is a number $\lambda >1$
and a pair of transverse measured foliations $(F^s,\mu^s)$ and $(F^u,\mu^u)$
such that $f(F^s,\mu^s)=(F^s,\frac{1}{\lambda}\mu^s)$ and $f(F^u,\mu^u)=(F^u,\lambda\mu^u)$.
The mapping torus $T_f$ of $f:X\rightarrow X$ is the space obtained from $X\times [0 ,1]$
 by identifying $(x,1)$ with $(f(x),0)$ for all $x\in X$. It is often more convenient to regard $T_f$
as the space obtained from $X\times [0,\infty )$ by identifying $(x,s+1)$ with $(f(x),s)$ for
all $x\in X,s\in [0 ,\infty )$. On $T_f$ there is a natural semi-flow
$\phi :T_f\times [0,\infty )\rightarrow T_f, \phi_t(x,s)=(x,s+t)$ for all $t\geq 0$.
Then the map  $f:X\rightarrow X$ is the return map of the semi-flow $\phi $.
A point $x\in X$ and a positive number $\tau >0$ determine the orbit
curve $\phi _{(x,\tau )}:={\phi_t(x)}_{0\leq t \leq \tau}$ in $T_f$.
The fixed points and periodic points of $f$ then correspond to closed orbits of various periods.
Take the base point $x_0$ of $X$ as the base point of $T_f$. According to van Kampen's Theorem,
the fundamental group $G :=\pi_ 1(T_f,x_0)$ is obtained from $\pi $ by adding a new
generator $z$ and adding the relations $z^gz^{-1}=\tilde f_*(g)$ for all $g\in \pi =\pi _1(X,x_0)$,
where $z$ is the generator of $\pi_1(S^1,x_0)$.
This means that $G$ is a semi-direct product $G=\pi \rtimes  Z$ of $\pi$ with $Z$.

There is a canonical projection
 $\tau : T_f \to R/Z$ given by
 $(x,s)\mapsto s$.
This induces a map
 $\pi_1(\tau):G=\pi_1(T_f,x_0)\to Z$.

 The Reidemeister number $R(f)$
 is equal to the number of homotopy
 classes of closed paths $\gamma$ in $T_f$
 whose projections onto $R/Z$
 are homotopic to the path
 $$
\begin{array}{cccl}
 \sigma :&[0,1] & \to     & $ $ \!\!\! R/Z \\
         &  s   & \mapsto & s.
\end{array}
 $$

Corresponding to this,  there is a group-theoretical interpretation of
 $R(f)$ as the number of
 usual conjugacy classes of elements $\gamma\in\pi_1(T_f)$
 satisfying $\pi_1(\tau)(\gamma) = z$.

 \begin{lemma}\cite{th, otal}\label{th}
The interior of  the mapping torus $ Int(T_f)$ admits a hyperbolic structure of finite volume if and only if
$f$ is isotopic to a pseudo-Anosov homeomorphism.
\end{lemma}

So, if the surface $X$ is closed and $f$ is isotopic to a pseudo-Anosov homeomorphism,
the mapping torus $T_f$ can be realised as a hyperbolic 3-manifold ,  $H^3/G$, where $H^3$
is the Poincare upper half space $\{(x,y,z): z  > 0, (x,y) \in R^2\}$ with the metric
$ds^2=(dx^2 +dy^2 + dz^2)/z^2$. 
The closed geodesics on a hyperbolic manifold are in one-to-one correspondence with
the free homotopy classes of loops. These classes of loops are in one-to-one correspondence with the
conjugacy classes of loxodromic elements in the fundamental group of the hyperbolic manifold.
This correspondence allowed  Ch. Epstein (see \cite{eps}, p.127) to study the asymptotics of such functions 
as   $p_n(x)$= \#\{primitive closed geodesics of length less than $x$ represented by an element of the
form $gz^n$\}  using the Selberg trace formula. A primitive closed geodesic is one which is not
an iterate of another closed geodesic. Later, Phillips and Sarnak \cite{ps} generalised results of Epstein 
 and obtained  for $n$-dimensional hyperbolic manifold the asymptotic of the number
of primitive closed geodesics of length at most $x$ lying in fixed homology class. 
The proof of this result makes  routine use of the Selberg trace formula.  In the more general case
of variable negative curvature, such asymptotics were obtained by Pollicott and Sharp \cite{pos}.
They used  a dynamical approach based on the geodesic flow.
We will only need an asymptotic   for $p_1(x)$.  Note  that closed geodesics represented by an
element of the form $gz$ are  automatically primitive, because they wrap
exactly once around the mapping torus (once around  the generator $z$). 
We have following asymptotic expansion \cite{eps, ps, pos}
\begin{equation}
p_1(x)=  \frac{e^{hx}}{x^{3/2}}( \sum_{n=0}^{N}\frac{C_n}{x^{n/2}} + o(\frac{1}{x^{N/2}}) ),
\end{equation}
for any $N>0$,
where $h=\dim T_f  - 1=2$ is the topological entropy of the geodesic flow on the unit-tangent bundle
$ST_f$, and the constant  $C_0 > 0$ depends on the volume of  hyperbolic 3-manifold $T_f$.
N. Anantharaman \cite{an} has shown that  the constants $C_n$ vanish if $n$ is odd.
So, we have the following  leading asymptotic behaviour:
\begin{equation}
p_1(x) \sim  C_0 \frac{e^{hx}}{x^{3/2}},  \mbox {as}  x \rightarrow \infty
\end{equation}

Notation. We write $f(x)\sim g(x)$  if  $\frac{f(x)}{g(x)}\rightarrow 1 $  as  $ x \rightarrow \infty $.

Now, using  the one-to-one correspondences  in Lemma \ref{lem:HillRanbij}  and   Reidemeister bijection in Itroduction
we  define  the  norm  of the 
lifting class, or of the corresponding  twisted
conjugacy class  $\{g\}_{\tilde f_*}$ in the fundamental group
of the surface $\pi =\pi _1(X, x_0)$,  as the length of the primitive closed geodesic  $\gamma $  on  $T_f$,
 which is represented by an element of the form $gz$. So, for example, the norm function $l^*$
on the set of twisted  conjugacy classes equals  $l^*=l\circ B$,
where $l$ is  length function on geodesics ($l(\gamma)$ is the length of the
primitive closed geodesic  $\gamma$) and a norm map $B$ is a bijection defined by formula $B(g)= gz$
between  the fundamental group
of  the surface $\pi =\pi _1(X, x_0)$  and the first coset $\pi _1(X, x_0)z$  in the fundamental group $G :=\pi_ 1(T_f,x_0)=\pi \rtimes  Z$.
\begin{lemma}
1) $ B(\gamma g \tilde f_*((\gamma)^{-1})=\gamma B(g)\gamma^{-1}$ 

2)  Function $l^*=l\circ B$  is a  non negative  twisted class function on G
\end{lemma}

\begin{proof} 
1)  $ B(\gamma g \tilde f_*((\gamma)^{-1})=\gamma g \tilde f_*(\gamma^{-1})z =\gamma g z\gamma^{-1} z^{-1}z=\gamma g z\gamma^{-1}=\gamma B(g)\gamma^{-1}$

2)  $l^*(\gamma g \tilde f_*((\gamma)^{-1})=l\circ B(\gamma g \tilde f_*((\gamma)^{-1})=l(\gamma B(g)\gamma^{-1})=l(B(g))= l^*(g)$
\end{proof}

In \cite{felas2} we introduced  the following  counting functions:\\
Tw($x$)= \#\{twisted conjugacy  classes for  $\tilde f_*$  in the fundamental group of surface  of norm  less than $x$\},\\
L($x$)= \# \{lifting  classes of $f$ of norm  less than $x$\}.
A norm function on the set of liftings of $f$ equals
  $l^{**}=l^*\circ L$, where $l^*$ is defined as above and $L( \tilde f^{\prime} =\alpha\circ\tilde f)=\alpha$ is a  function from the set of liftings of $f$ to the group of their coordinates
$\pi =\pi _1(X, x_0)$  (see the introduction). The function $l^{**}$ is a class
function on the set of liftings,  constant on the liftings classes.

\begin{theorem}\cite{felas2}
Let $X$ be a closed surface of negative Euler characteristic and let $f:X\rightarrow X$
be a pseudo-Anosov homeomorphism. Then
$$ L(x)=Tw(x)= 
 \frac{e^{2x}}{x^{3/2}}( \sum_{n=0}^{N}\frac{C_n}{x^{n/2}}  + o(\frac{1}{x^{N/2}}) ),$$
where the constant  $C_0 > 0$
depends on the volume of  the  hyperbolic 3-manifold  $T_f$, and the
constants $C_n$ vanish if $n$ is odd.
\end{theorem}

 Proof.
The proof follows from Reidemeister bijection,  Lemma \ref{th},  Lemma \ref{lem:HillRanbij} and   the asymptotic expansion (3.1).
\begin{corollary}
For pseudo-Anosov homeomorphisms  of  closed surfaces  the Reidemeister
number is infinite.
\end{corollary}

\begin{que}
How can one define the norm of a  twisted conjugacy class
in the general case?

\end{que}

\section{Nielsen fixed point theory and symplectic Floer homology} \label{sec:floer}
In the  dimension two a diffeomorphism is symplectic if it preserves
area. As  a consequence, the symplectic geometry of surfaces lacks many of the interesting phenomena which are encountered in higher dimensions. For  example, two symplectic automorphisms of a closed surface are symplectically isotopic iff they are homotopic, by a theorem of Moser\cite{Mo}.
On other hand symplectic fixed point theory is very nontrivial in dimension 2,
as it is  shown by the Poincare-Birkhoff theorem.
It is known that symplectic Floer homology on surface
is a simple model for the instanton Floer homology of the
  mapping torus of the surface diffeomorphism \cite{S}.

\subsection{ Symplectic Floer homology}

\subsubsection{Monotonicity}

In this section we discuss the notion of monotonicity as defined in \cite{S,G}.
Monotonicity plays important role for  Floer homology in two
dimensions.
Throughout this article, $M$ denotes a closed connected and oriented
2-manifold of genus $\geq2$.   Pick an everywhere positive two-form $\omega$
on $M$.

Let $\phi\in\symp(M,\omega)$, the group of  symplectic automorphisms
 of the two-dimensional symplectic
manifold $(M,\omega)$.
The mapping torus of $\phi$,  $T_\phi = \R\times M/(t+1,x)\sim(t,\phi(x)),$
is a 3-manifold fibered over $S^1=\R/\Z$.
There are two natural second
cohomology classes on $T_\phi$, denoted by $[\ophi]$ and $c_\phi$. The first one is
represented by the closed two-form $\ophi$ which is induced from the pullback
of $\omega$ to $\R\times M$. The second is the Euler class of the
vector bundle
$
V_\phi =
\R\times T M/(t+1,\xi_x)\sim(t,\de\phi_x\xi_x),
$
which is of rank 2 and inherits an orientation from $TM$.

$\phi\in\symp(M,\omega)$ is called {\bf monotone}, if
$
[\omega_\phi] = (\area_\omega(M)/\chi(M))\cdot c_\phi
$
in $H^2(T_\phi;\R)$; throughout this article
$\symp^m(M,\omega)$ denotes the set of
monotone symplectomorphisms.

Now $H^2(T_\phi;\R)$ fits into the following short exact sequence \cite{S,G}
\begin{equation}\label{eq:cohomology}
0 \To  \frac{H^1(M;\R)}{\im(\id-\phi^*)}
\stackrel{d}{\To} H^2(T_\phi;\R)
\stackrel{r^*}{\To} H^2(M;\R),
\To 0.
\end{equation}
where the map $r^*$  is restriction to the fiber.
The map $d$ is defined as follows.
Let $\rho:I\to\R$ be a smooth function which vanishes near $0$ and $1$ and
satisfies $\int_0^1\!\rho\,\de t=1$.
If $\theta$ is a closed 1-form on $M$, then
$\rho\cdot\theta\wedge\de t$ defines a closed 2-form on $T_\phi$; indeed
$
d[\theta] = [\rho\cdot\theta\wedge\de t].
$
The map $r:M\inclusion T_\phi$ assigns to each $x\in M$ the
equivalence class of $(1/2,x)$.
Note, that $r^*\ophi=\omega$ and $r^*c_\phi$ is the Euler class of
$TM$.
Hence, by \eqref{eq:cohomology}, there exists a unique class
$m(\phi)\in H^1(M;\R)/\im(\id-\phi^*)$ satisfying
$
d\,m(\phi) = [\ophi]-(\area_\omega(M)/\chi(M))\cdot c_\phi,
$
where $\chi$ denotes the Euler characteristic.
Therefore, $\phi$ is monotone if and only if $m(\phi)=0$.
\smallskip\\
We recall the fundamental properties of $\symp^m(M,\omega)$ from \cite{S,G}.
Let  $\diff^+(M)$ denotes  the group of orientation preserving
diffeomorphisms of $M$.

(Identity) $\id_M \in \symp^m(M,\omega)$.

(Naturality)
If $\phi\in\symp^m(M,\omega),\psi\in\diff^+(M)$, then
$\psi^{-1}\phi\psi\in\symp^m(M,\psi^*\omega)$.

(Isotopy)
Let $(\psi_t)_{t\in I}$ be an isotopy in $\symp(M,\omega)$, i.e. a smooth
path with $\psi_0=\id$.
Then
$$
m(\phi\circ\psi_1)=m(\phi)+[\flux(\psi_t)_{t\in I}]
$$
in $H^1(M;\R)/\im(\id-\phi^*)$; see \cite[Lemma 6]{S}.
For the definition of the flux homomorphism see \cite{MS}.

(Inclusion)
The inclusion $\symp^m(M,\omega)\inclusion\diff^+(M)$ is a homotopy
equivalence.
\smallskip\\
(Floer homology)
To every $\phi\in\symp^m(M,\omega)$ symplectic Floer homology
theory assigns a $\Z_2$-graded vector space $\HF(\phi)$ over $\Z_2$, with an
additional multiplicative structure, called the quantum cap product,
$
H^*(M;\Z_2)\otimes\HF(\phi)\To\HF(\phi).
$
For $\phi=\id_M$ the symplectic Floer homology $\HF(\id_M)$ are  canonically isomorphic to ordinary homology  $H_*(M;\Z_2)$ and quantum cap product agrees with the ordinary cap product.
Each $\psi\in\diff^+(M)$ induces an isomorphism
$\HF(\phi)\cong\HF(\psi^{-1}\phi\psi)$ of $H^*(M;\Z_2)$-modules.

(Invariance)
If $\phi,\phi'\in\symp^m(M,\omega)$ are isotopic, then
$\HF(\phi)$ and $\HF(\phi')$ are naturally isomorphic as
$H^*(M;\Z_2)$-modules.

This is proven in \cite[Page 7]{S}. 
Note that every Hamiltonian perturbation
of $\phi$ (see \cite{ds}) is also in $\symp^m(M,\omega)$.
\smallskip\\
Now let $g$ be a mapping class of $M$, i.e. an isotopy class of $\diff^+(M)$.
Pick an area form $\omega$ and a
representative $\phi\in\symp^m(M,\omega)$ of $g$.
Then $\HF(\phi)$ is an invariant of $g$, which is
denoted by $\HF(g)$. Note that $\HF(g)$ is independent of the choice of an area
form $\omega$ by Moser's isotopy theorem \cite{Mo} and naturality of Floer homology.

\subsubsection{Floer homology}

Let $\phi\in\symp(M,\omega)$.There are two ways of constructing Floer
homology detecting its fixed points, $\Fix(\phi)$. Firstly, the graph of $\phi$
is a Lagrangian submanifold of $M\times M,(-\omega)\times\omega)$ and its fixed points correspond to the intersection points of graph($\phi$) with the
diagonal $\Delta=\{(x,x)\in M\times M\}$. Thus we have the Floer homology of the Lagrangian intersection $\HF(M\times M,\Delta, graph (\phi))$.
This intersection is transversal if the fixed points of $\phi$ are nondegenerate, i.e. if 1 is not an eigenvalue of $d\phi(x)$, for $x\in\Fix(\phi)$.
The second approach was mentioned by Floer in \cite{Floer1} and presented with
details by Dostoglou and Salamon  in \cite{ds}.We follow here  Seidel's approach \cite{S} which, comparable with \cite {ds}, uses a larger
class of perturbations, but such that the perturbed action form is still
cohomologous to the unperturbed. As a consequence, the usual invariance of
Floer homology under Hamiltonian isotopies is extended to the stronger
property stated above.
Let now  $\phi\in\symp^m(M,\omega)$, i.e  $\phi$ is  monotone.
Firstly, we give  the definition of $\HF(\phi)$  in the special case
where all the fixed points of $\phi$ are non-degenerate, i.e. for all $y\in\fix(\phi)$, $\det(\id-\de\phi_y)\ne0$, and then
following Seidel´s approach  \cite{S} we consider general case when
$\phi$ has degenerate fixed points.
 Let $\Of = \{ y \in C^{\infty}(\R,M)\,|\, y(t) = \phi(y(t+1)) \}$ be the twisted free loop space, which is also
the space of sections of $T_\phi \rightarrow S^1$. The action form is the
closed one-form $\alpha_\phi$ on $\Of$ defined by
$
\alpha_\phi(y) Y = \int_0^1 \omega(dy/dt,Y(t))\,dt,
$
where $y\in\Of$ and $Y\in T_y\Of$,
 i.e. $Y(t)\in T_{y(t)}M$
 and
$Y(t)=\de\phi_{y(t+1)}Y(t+1)$ for all  $t\in\R$.

The tangent bundle of any symplectic manifold admits an almost complex structure $ J:TM\To TM$ which is compatible with $\omega$ in sense that $(v,w)=\omega(v,Jw)$ defines a Riemannian metric.
 Let $J=(J_t)_{t \in \R}$
be a smooth path of $\omega$-compatible almost  complex structures on
$M$ such that $J_{t+1}=\phi^*J_t$.
If $Y,Y'\in T_y\Of$, then
$\int_0^1\omega(Y'(t),J_t Y(t))\de t$ defines a metric
on the loop space $\Of$. So  the critical points of $\alpha_\omega$ are the constant paths in  $\Of$ and  hence the fixed points of $\phi$. The negative gradient lines of $\alpha_\omega$ with respect to the
metric above  are solutions of the partial differential equations
with boundary conditions
\begin{equation}\label{eq:corbit}
\left\{\begin{array}{l}
u(s,t) = \phi(u(s,t+1)), \\
\p_s u + J_t(u)\p_t u = 0, \\
\lim_{s\to\pm\infty}u(s,t) \in \Fix(\phi)
\end{array}\right.
\end{equation}
These are exactly  Gromov's pseudoholomorphic  curves \cite{Gromov}.

For $y^\pm\in\fix(\phi)$, let $\M(y^-,y^+;J,\phi)$ denote the space of smooth maps $u:\R^2\to M$ which satisfy the  equations \eqref{eq:corbit}.
Now to every $u\in\M(y^-,y^+;J,\phi)$ we  associate a Fredholm operator
$\De_u$ which linearizes (\ref{eq:corbit}) in suitable Sobolev spaces. The
index of this operator is given by the so called Maslov index $\mu(u)$,
which satisfies $\mu(u)=\deg(y^+)-\deg(y^-)\text{ mod }2$, where  $(-1)^{\deg y}=\sign(\det(\id-\de\phi_y))$. We have no bubbling, since for surface
$\pi_2(M)=0$.

 For a generic
$J$, every $u\in\M(y^-,y^+;J,\phi)$ is regular, meaning that $\De_u$ is onto.
Hence, by the implicit function theorem, $\M_k(y^-,y^+;J,\phi)$ is
a smooth $k$-dimensional manifold, where $\M_k(y^-,y^+;J,\phi)$ denotes the
subset of 
those $u\in\M(y^-,y^+;J,\phi)$ with $\mu(u)=k\in\Z$.
Translation of the $s$-variable defines a free $\R$-action on 1-dimensional
manifold $\M_1(y^-,y^+;J,\phi)$ and hence the quotient is a discrete set of points.

 The energy of a map $u:\R^2\to M$ is given by 
$$
E(u) = \int_{\R}\int_0^1 \omega\big(\p_tu(s,t),J_t\p_tu(s,t)\big)\,\de t\de s
$$
  for all $y\in\fix\phi$.  P.Seidel has proved  in \cite{S} that if $\phi$ is monotone, then the energy is constant on each $\M_k(y^-,y^+;J,\phi)$.
Since all fixed points of $\phi$ are nondegenerate the set  $\fix\phi$ is a finite set and the
$\Z_2$-vector space  $
\CF(\phi) := \Z_2^{\#\fix\phi}$
admits a $\Z_2$-grading with $(-1)^{\deg y}=\sign(\det(\id-\de\phi_y))$,
for all $y\in \fix\phi$.
The boundness of the energy  $E(u)$   for monotone  $\phi$  implies that   the  0-dimensional  quotients   $\M_1(y_-,y_+,J,\phi)/\R$   are actually finite sets. Denoting by  $n(y_-,y_+)$  the number of
points mod 2 in each of them, one defines a differential  $\partial_{J}:
CF_*(\phi) \rightarrow CF_{* + 1}(\phi)$  by $\partial_{J}y_- =
\sum_{y_+} n(y_-,y_+) {y_+}$.  Due to gluing theorem  this Floer boundary operator satisfies  $\partial_{J} \circ
\partial_{J} = 0$.  For gluing  theorem to hold one needs again the  boundness of the energy $E(u)$ .  It follows that  $ (\CF(\phi),\partial_{J})$  is a chain complex  and its homology is by definition the Floer homology of  $\phi$   denoted $HF_*(\phi)$. It  is independent of $J$ and is an invariant of $\phi$.

If $\phi$ has degenerate fixed points one needs to perturb equations
\eqref{eq:corbit} in order to define the Floer homology. Equivalently, one
could say that the action form needs to be perturbed.
 The necessary analysis is  given in \cite{S}, it  is essentially  the same as in the slightly different
situations considered in \cite{ds}. But  Seidel's approach also differs from the usual one in \cite{ds}. He uses a larger
class of perturbations, but such that the perturbed action form is still
cohomologous to the unperturbed.

\subsection{Nielsen numbers and Floer homology\label{sec:numbers}}

\subsubsection{Periodic diffeomorphisms}

\begin{lemma}\label{lemma:jiang}\cite{J}
Let $\phi$  a non-trivial orientation
preserving periodic diffeomorphism of a compact connected surface $M$
of Euler characteristic $\chi(M) \leq0$. Then each fixed point class of $\phi$
consists of a single point.
\end{lemma}
There are  two criteria for monotonicity which we use later on.
 Let $\omega$ be an area form on $M$ and
$\phi\in\symp(\Sigma,\omega)$. 
\begin{lemma}\label{lemma:monotone1}\cite{G}
Assume that every class $\alpha\in\ker(\id-\phi_*)\subset H_{1}(M;\Z)$ is
represented by a map $\gamma:S\to\fix\phi$, where $S$ is a compact oriented
1-manifold. Then $\phi$ is monotone.
\end{lemma}
\begin{lemma}\label{lemma:monotone2}\cite{G}
If $\phi^k$ is monotone for some $k>0$, then $\phi$ is monotone.
If $\phi$ is monotone, then $\phi^k$ is monotone for all $k>0$. 
\end{lemma}

We shall say  that $\phi:M\rightarrow M$ is a  periodic map of  period $m$,
if $\phi^m$ is  the identity map $\id_M:M\rightarrow M$.

\begin{theorem}\label{thm:main}\cite{ff}
If $\phi$ is  a non-trivial orientation
preserving periodic diffeomorphism of a compact connected surface $M$
of Euler characteristic $\chi(M) <  0$, then $\phi$ is monotone with
respect to some $\phi$-invariant area form and
$$
\HF(\phi) \cong
\Z_2^{N (\phi)},  \,\,\,  \dim\HF(\phi)= N (\phi),
$$
where  $ N (\phi)$ denotes the Nielsen  number of  $\phi$.
\end{theorem}

\begin{proof}
Let  $\phi$ be  a periodic diffeomorphism  of least period $l$.
 First note that
if $\tilde{\omega}$ is an area form on $M$, then area form
$\omega:=\sum_{i=1}^\ell(\phi^i)^*\tilde{\omega}$ is 
 $\phi$-invariant, i.e. $\phi\in\symp(M,\omega)$. By periodicity of  $\phi$,
 $\phi^l$ is  the identity map $\id_M:M\rightarrow M$.
Then from Lemmas \ref{lemma:monotone1}  and \ref{lemma:monotone2} it follows that $\omega$ can be chosen such that $\phi\in\symp^m(M,\omega)$.

 Lemma \ref{lemma:jiang}  implies that every
$y\in\fix\phi$ forms a different fixed point class of
$\phi$, so $ \#\fix\phi= N (\phi)$.
This has an immediate consequence for the Floer complex $(\CF(\phi),\p_{J})$
with respect to a generic $J=(J_t)_{t\in\R}$.
If $y^\pm\in\fix\phi$ are in different fixed point classes, then
$\M(y^-,y^+;J,\phi)=\emptyset$. This follows from the first equation in
\eqref{eq:corbit}.
Then the boundary map in  the Floer complex is zero $\partial_{J}=0$ and
$\Z_2$-vector space  $
\CF(\phi) := \Z_2^{\#\fix\phi}=\Z_2^{N(\phi)}$. This immediately  implies
$
\HF(\phi) \cong
\Z_2^{N (\phi)}
$ and $ \dim\HF(\phi)= N (\phi)$.

\end{proof}

\subsubsection{Algebraically finite mapping classes}

 A mapping class of $M$ is called  algebraically finite if it
does not have any pseudo-Anosov components in the sense of Thurston's
theory of surface diffeomorphism.The term algebraically finite goes back to J. Nielsen.\\
 In  \cite{G}  the diffeomorphisms  of finite type  were defined . These are
special representatives of algebraically finite mapping classes
adopted to the symplectic geometry. 
\begin{dfn}\label{def:ftype}\cite{G}
We call $\phi\in\diff_+(M)$ of {\bf finite type} if the following holds.
There is a $\phi$-invariant finite union $N\subset M$ of disjoint
non-contractible annuli such that:
\smallskip\\
(1) $\phi|M \setminus N$ is periodic, i.e. there exists
$\ell>0$ such that $\phi^\ell|M \setminus N=\id$.
\smallskip\\
(2) Let $N'$ be a connected component of $N$ and $\ell'>0$ be the
smallest integer such that $\phi^{\ell'}$ maps $N'$ to itself. Then
$\phi^{\ell'}|N'$ is given by one of the following two models with respect to
some coordinates $(q,p)\in I\times S^1$:
\medskip\\
\begin{minipage}{4cm}
(twist map)
\end{minipage}
\begin{minipage}{6cm}
$(q,p)\Mapsto(q,p-f(q))$
\end{minipage}
\medskip\\
\begin{minipage}{4cm}
(flip-twist map)
\end{minipage}
\begin{minipage}{6cm}
$(q,p)\Mapsto(1-q,-p-f(q))$,
\end{minipage}
\medskip\\
where $f:I\to \R$ is smooth and strictly monotone.
A twist map is called  positive or negative,
if $f$ is increasing or decreasing.
\smallskip\\
(3) Let $N'$ and $\ell'$ be as in (2).
If $\ell'=1$ and $\phi|N'$ is a twist map, then $\im(f)\subset[0,1]$,
i.e. $\phi|\text{int}(N')$ has no fixed points.
\smallskip\\
(4) If two connected components of $N$ are homotopic, then the corresponding local
models of $\phi$ are either both positive or both negative twists.

\end{dfn}
The term flip-twist map is taken from \cite{JG}.

By  $M_{\id}$ we denote the union of the components
of $M\setminus\text{int}(N)$, where $\phi$ restricts to the identity.

The next  lemma describes the set of fixed point classes of $\phi$. It
is a special case of a theorem by B.~Jiang and J.~Guo~\cite{JG}, which gives
for any mapping class a representative that realizes its
Nielsen number.
\begin{lemma}[Fixed point classes]\label{lemma:fclass}\cite{JG}
Each fixed point class of $\phi$ is either a connected component of $M_{\id}$ or
consists of a single fixed point. A fixed point $x$ of the second type satisfies
$\det(\id-\de\phi_x)>0$.
\end{lemma}

The monotonicity of diffeomorphisms of finite type
was  investigated in details in \cite{G}.
Let $\phi$ be a diffeomorphism of finite type and
$\ell$ be as in (1).
Then $\phi^\ell$ is the product of (multiple)  Dehn twists along $N$.
Moreover, two parallel Dehn twists have the same sign, by (4). We say that
$\phi$ has  uniform twists, if $\phi^\ell$ is the product of only
positive, or only negative Dehn twists.
\smallskip\\
Furthermore, we denote by $\ell$ the smallest positive integer such that
$\phi^\ell$ restricts to the identity on $M\setminus N$.

If $\omega'$ is an area form on $M$ which is the standard form
$\de q\wedge\de p$ with respect to the $(q,p)$-coordinates on $N$, then
$\omega:=\sum_{i=1}^\ell(\phi^i)^*\omega'$ is
standard on $N$ and $\phi$-invariant, i.e. $\phi\in\symp(M,\omega)$.
To prove that $\omega$ can be chosen such that $\phi\in\symp^m(M,\omega)$,
Gautschi distinguishes two cases: uniform and non-uniform twists. In the first case he proves the following stronger statement.
\begin{lemma}\label{lemma:monotone3}\cite{G}
If $\phi$ has uniform twists and $\omega$ is a $\phi$-invariant area
form, then $\phi\in\symp^m(M,\omega)$.
\end{lemma}

In the non-uniform case, monotonicity  does not
hold for arbitrary $\phi$-invariant area forms.
\begin{lemma}\label{lemma:monotone4}\cite{G}
If $\phi$ does not have uniform twists, there exists a $\phi$-invariant
area form $\omega$ such that $\phi\in\symp^m(M,\omega)$. Moreover,
$\omega$ can be chosen such that it is the standard form $\de q\wedge\de p$ on
$N$.
\end{lemma}

\begin{theorem}\label{thm:main0}\cite{ff}
If $\phi$ is  a  diffeomorphism of finite type  of a compact connected surface $M$
of Euler characteristic $\chi(M) <  0$ and if  $\phi$ has  only isolated fixed points , then $\phi$ is monotone with
respect to some $\phi$-invariant area form and
$$
\HF(\phi) \cong
\Z_2^{N (\phi)},  \,\,\,  \dim\HF(\phi)= N (\phi),
$$
where  $ N (\phi)$ denotes the Nielsen  number of  $\phi$.
\end{theorem}
\begin{proof}
From Lemmas \ref{lemma:monotone3}  and \ref{lemma:monotone4} it follows that $\omega$ can be chosen such that $\phi\in\symp^m(M,\omega)$.
 Lemma \ref{lemma:fclass}  implies that every
$y\in\fix\phi$ forms a different fixed point class of
$\phi$, so $ \#\fix\phi= N (\phi)$.
This has an immediate consequence for the Floer complex $(\CF(\phi),\p_{J})$
with respect to a generic $J=(J_t)_{t\in\R}$.
If $y^\pm\in\fix\phi$ are in different fixed point classes, then
$\M(y^-,y^+;J,\phi)=\emptyset$. This follows from the first equation in
\eqref{eq:corbit}.
Then the boundary map in  the Floer complex is zero $\partial_{J}=0$ and
$\Z_2$-vector space  $
\CF(\phi) := \Z_2^{\#\fix\phi}=\Z_2^{N(\phi)}$. This immediately  implies
$
\HF(\phi) \cong
\Z_2^{N (\phi)}
$ and $ \dim\HF(\phi)= N (\phi)$.

\end{proof}

\begin{theorem}\label{th:ga}\cite{G}
Let  $\phi$ be  a diffeomorphism of finite type, then $\phi$ is monotone with
respect to some $\phi$-invariant area form and
$$
\HF(\phi) =
H_*(M_{\id},\p_{M_{\id}};\Z_2)\oplus
\Z_2^{L(\phi|M\setminus M_{\id})}.
$$
Here, $L$ denotes the Lefschetz number.
\end{theorem}

\begin{proof}

The main idea of the proof  is a separation mechanism for Floer connecting
orbits.
Together with the topological separation of fixed points discussed in
theorem \ref{thm:main0} , it allows us to compute the Floer homology of
diffeomorphisms of finite type.

There exists a function $H: M\to\R$ 
such that $H|\text{int}(M_{id})$ is a Morse function,
meaning that all the critical points are non-degenerate and $H|M\setminus M_{id}=0$. 
Let $(\psi_t)_{t\in\R}$ denote the Hamiltonian flow generated by $H$ with
respect to the fixed area form $\omega$ and set
$
\Phi:=\phi\circ\psi_1.$ Then $
\fix\Phi = \big(\crit(H)\cap M_{id}\big)\cup\big(\fix\phi\setminus M_{id}\big).
$
In particular, $\Phi$ only has non-degenerate fixed points.
Let $N_0\subset M_{id}$ be a collar neighborhood of
$\p M_{id}$.
Let $x^-,x^+\in\fix\Phi\cap M_{id}$ be in the same connected component of
$M_{id}$.
If $u\in\M(x^-,x^+;J,\Phi)$, then $\im u \subset \Sigma_\delta$,
where $\Sigma_\delta$ denotes the $\delta$-neighborhood of
$M_{id}\setminus N_0$  with respect to any of the metrics $\omega(.,J_t.)$ 
\cite{S,G}

 Moreover, lemma \ref{lemma:fclass}  implies that every
$y\in\fix\phi\setminus M_{id}$ forms a different fixed point class of
$\Phi$.
This has an immediate consequence for the Floer complex $(\CF(\Phi),\p_{J})$
with respect to a generic $J=(J_t)_{t\in\R}$. Namely,
$(\CF(\Phi),\p_{J})$ splits into the subcomplexes $(\Cc_1,\p_1)$ and
$(\Cc_2,\p_2)$, where $\Cc_1$ is generated by $\crit(H)\cap M_{id}$
and $\Cc_2$ by $\fix\phi\setminus M_{id}$. Moreover, $\Cc_2$ is graded by~$0$
and $\p_2=0$ \cite{G}.The homology of $(\Cc_1,\p_1)$ is isomorphic to
$H_*(M_{id},\p_+M_{id};\Z_2)$ \cite{S,G}. So
$
\HF(\phi) \cong
H_*(M_{id},\p_+\Sigma_0;\Z_2)\oplus
\Z_2^{\#\fix\phi|M\setminus M_{id}}.
$
Since every fixed point of $\phi|M\setminus M_{id}$ has fixed point
index~1, the Lefschetz fixed point formula implies that
$
\#(\fix\phi\setminus M_{id}) = \Lambda(\phi|M\setminus M_{id}).
$

\end{proof}

\begin{rk}
 In  the  theorem \ref{thm:main0} the  set $ M_{\id}$ is empty and  every fixed point of $\phi$ has fixed point index 1 \cite{JG}.
The Lefschetz fixed point formula implies that
$\#\fix \phi=N(\phi)=L(\phi)$ .
So, theorem \ref{thm:main0}  follows also  from theorem  \ref{th:ga}.

\end{rk}

\subsection{Symplectic zeta functions and asymptotic invariant\label{sec:zeta}}

\subsubsection{Symplectic zeta functions}
Let $\Gamma = \pi_0(Diff^+(M))$ be the mapping class group of a closed connected oriented surface $M$ of genus $\geq 2$. Pick an everywhere positive two-form $\omega$ on $M$. A isotopy  theorem of Moser \cite{Mo} says that each  mapping class of  $g \in \Gamma$,  i.e. an isotopy class of $Diff^+(M)$,  admits representatives which preserve $\omega$. Due  to Seidel\cite{S} we can
pick  a monotone  representative $\phi\in\symp^m(M,\omega)$ of $g$.
Then $\HF(\phi)$ is an invariant of $g$, which is
denoted by $\HF(g)$. Note that $\HF(g)$ is independent of the choice of an area
form $\omega$ by Moser's  theorem  and naturality of Floer homology.
By  lemma  \ref{lemma:monotone2}  symplectomorphisms  $\phi^n$
are  also monotone for all $n>0$.
Taking a dynamical point of view,
 we consider the iterates of monotone symplectomorphism  $\phi$
and  define the first symplectic  zeta function of $\phi$ \cite{ff}
 as the following power series:
$
 \chi_\phi(z) =
 \exp\left(\sum_{n=1}^\infty \frac{\chi(\HF(\phi^n))}{n} z^n \right),
$
where $\chi(\HF(\phi^n))$ is the Euler characteristic of Floer homology group
of $\phi^n$.
Then $ \chi_\phi(z)$ is an invariant of $g$, which we
denote by $ \chi_g(z)$.
Let us consider  the  Lefschetz  zeta function
 $
  L_\phi(z) := \exp\left(\sum_{n=1}^\infty \frac{L(\phi^n)}{n} z^n \right),
$
   where
 $
   L(\phi^n) := \sum_{k=0}^{\dim X} (-1)^k \tr\Big[\phi_{*k}^n:H_k(M;\Q)\to H_k(M;\Q)\Big]
 $
 is the Lefschetz number of $\phi^n$. 
\begin{theorem}\label{thm:lef}\cite{ff}
Symplectic zeta function $ \chi_\phi(z)$ is a rational function of $z$ and
$$
 \chi_\phi(z)= L_\phi(z)=\prod_{k=0}^{\dim X}
          \det\big(I-\phi_{*k}.z\big)^{(-1)^{k+1}}.
$$
\end{theorem}

\begin{proof} 
If for every $n$ all the fixed points of $\phi^n$ are non-degenerate, i.e. for all $x\in\fix(\phi^n)$, $\det(\id-\de\phi^n(x))\ne0$, then we have( see section
\ref{sec:floer}): 
$\chi(\HF(\phi^n))=\sum_{x=\phi^n(x)} \sign(\det(\id-\de\phi^n(x)))=L(\phi^n).
$
If we have degenerate fixed points one needs   to perturb equations
\eqref{eq:corbit} in order to define the Floer homology.
 The necessary analysis is given in \cite{S}  is essentially  the same as in the slightly different situations considered in \cite{ds}, where the  above connection between the Euler characteristic and the Lefschetz number
 was firstly  established.
\end{proof}

In \cite{ff} we have  defined  the second symplectic zeta function for 
 monotone symplectomorphism  $\phi$  as the following power series:
$
 F_\phi(z) = 
 \exp\left(\sum_{n=1}^\infty \frac{\dim\HF(\phi^n)}{n} z^n \right).
$
Then $ F_\phi(z)$ is an invariant of mapping class  $g$, which we
denote by $ F_g(z)$.

Motivation for this definition was  the  theorem \ref{thm:main} and nice analytical
properties of the Nielsen zeta function $N_\phi(z) =
 \exp\left(\sum_{n=1}^\infty \frac{N(\phi^n)}{n} z^n \right)$, see \cite{FelshB}. 
We denote the  numbers  $\dim\HF(\phi^n) $ by $N_n$. Let $ \mu(d), d \in N$,
be the M\"obius function.

 \begin{theorem}\cite{ff}
Let $\phi$ be  a non-trivial orientation
preserving periodic diffeomorphism  of least period $m$ of a compact connected surface $M$
of Euler characteristic $\chi(M) < 0$  . Then the
  zeta function $ F_\phi(z)$ is  a radical of a rational function and 
$
 F_\phi(z) =\prod_{d\mid m}\sqrt[d]{(1-z^d)^{-P(d)}},
$
 where the product is taken over all divisors $d$ of the period $m$, and $P(d)$ is the integer
$  P(d) = \sum_{d_1\mid d} \mu(d_1)N_{d\mid d_1} .  $
\end{theorem}

\begin{rk}
Given a symplectomorphism $\phi$ of surface $M$, one can form
the symplectic mapping torus
$M^4_{\phi}=T^3_{\phi}\rtimes S^1$, where $T^3_{\phi}$ is  usual mapping torus
(see section \ref{asympt}).
Ionel and Parker \cite{IP} have computed the degree zero Gromov invariants(these are built from the invariants of Ruan and Tian)  of
$M^4_{\phi}$ and of fiber sums of the $M^4_{\phi}$ with other symplectic manifolds. This is done by expressing the Gromov invariants in terms of the
Lefschetz zeta function $ L_\phi(z)$ . The result is a large set of interesting non-Kahler
symplectic manifolds with computational ways of distinguishing them. In dimension four this gives a symplectic construction of the exotic elliptic surfaces of Fintushel and Stern.
In higher dimensions it gives many examples of manifolds which are diffeomorphic but not equivalent as symplectic manifolds.
  Theorem \ref{thm:lef} implies that the Gromov invariants of $M^4_{\phi}$ are related to symplectic Floer homology  of $\phi $ via zeta
function
 $\chi_\phi(z)= L_\phi(z)$. We hope that  the second symplectic zeta function $
 F_\phi(z)$ give rise to a new invariant of symplectic 4-manifolds.

\end{rk}

\subsubsection{Topological entropy and the Nielsen numbers}

   A basic relation between Nielsen numbers and topological entropy $h(f)$ was found by N. Ivanov
   \cite{i1}. We present here a very short proof of Jiang  of the Ivanov's  inequality.
\begin{lemma}\label{lemma:ent}\cite{i1}
$$
h(f) \geq \limsup_{n} \frac{1}{n}\cdot\log N(f^n)
 $$
\end{lemma}
{ Proof}
 Let $\delta$ be such that every loop in $X$ of diameter $ < 2\delta $ is contractible.
 Let $\epsilon >0$ be a smaller number such that $d(f(x),f(y)) < \delta $ whenever $ d(x,y)<2\epsilon $. Let $E_n \subset X $ be a set consisting of one point from each essential fixed point class of $f^n$. Thus $ \mid E_n \mid =N(f^n) $. By the definition of $h(f)$, it suffices
 to show that $E_n$ is $(n,\epsilon)$-separated.
 Suppose it is not so. Then there would be two points $x\not=y \in E_n$ such that $ d(f^i(x), f^i(y)) \leq \epsilon$ for $o\leq i< n$ hence for all $i\geq 0$. Pick a path $c_i$ from $f^i(x)$ to
 $f^i(y)$ of diameter $< 2\epsilon$ for $ o\leq i< n$ and let $c_n=c_0$. By the choice of $\delta$
 and $\epsilon$ ,  $f\circ c_i \simeq c_{i+1} $ for all $i$, so $f^n\circ c_0\simeq c_n=c_0$. such that
 This means $x,y$ in the same fixed point class of $f^n$, contradicting the construction of $E_n$.

 This inequality is remarkable in that it does not require smoothness of the map and provides a common lower bound for the topological entropy of all maps in a homotopy class.

We recall Thurston classification theorem for homeomorphisms of surfase $M$
of genus $\geq 2$.

\begin{theorem}\label{thm:thur}\cite{Th}
Every homeomorphism $\phi: M\rightarrow M $ is isotopic to a homeomorphism $f$
such that either\\
(1) $f$ is a periodic map; or\\
(2) $f$ is a pseudo-Anosov map, i.e. there is a number $\lambda >1$(stretching factor)
 and a pair of transverse measured foliations $(F^s,\mu^s)$ and $(F^u,\mu^u)$ such that $f(F^s,\mu^s)=(F^s,\frac{1}{\lambda}\mu^s)$ and $f(F^u,\mu^u)=(F^u,\lambda\mu^u)$; or\\
(3)$f$ is reducible map, i.e. there is a system of disjoint simple closed curves $\gamma=\{\gamma_1,......,\gamma_k\}$ in $int M$ such that $\gamma$ is invariant by $f$(but $\gamma_i$ may be permuted) and $ \gamma$ has a $f$-invariant
tubular neighborhood $U$  such that each component of $M\setminus U$  has negative
Euler characteristic and on each(not necessarily connected) $f$-component of
$M\setminus U$, $f$ satisfies (1) or (2).

\end{theorem}
The map $f$ above is called the Thurston canonical form of $f$. In (3) it
can be chosen so that some iterate $f^m$ is a generalised Dehn twist on $U$.
Such a $f$ , as well as the $f$ in (1) or (2), will be called standard.
A key observation is that if $f$ is standard, so are all iterates of $f$.

\begin{lemma}\label{lem:flp}\cite{flp} Let $f$ be a pseudo-Anosov  homeomorphism with stretching factor $\lambda >1$ of surfase $M$
of genus $\geq 2$. Then
$$h(f)=log(\lambda)= \limsup_{n} \frac{1}{n}\cdot\log N(f^n)$$
\end{lemma}

\begin{lemma}\label{lem:j1}\cite{j1}
Suppose $f$ is a standard homeomorphism  of surfase $M$
of genus $\geq 2$ and  $\lambda$ is the largest stretching factor  of the  pseudo-Anosov pieces( $\lambda=1$ if there is no pseudo-Anosov piece).
Then
$$h(f)=log(\lambda)= \limsup_{n} \frac{1}{n}\cdot\log N(f^n)$$
\end{lemma}

\subsubsection{Asymptotic invariant}

The growth rate of a sequence $a_n$ of complex numbers is defined
by
 $$
 \grow ( a_n):= max \{1,  \limsup_{n \rightarrow \infty} |a_n|^{1/n}\}
$$

which could be infinity. Note that $\grow(a_n) \geq 1$ even if all $a_n =0$.
When  $\grow(a_n) > 1$, we say that the sequence $a_n$ grows exponentially.
\begin{dfn}
 We define the asymptotic invariant $ F^{\infty}(g)$ of  mapping class   $g \in \Gamma = \pi_0(Diff^+(M))$ to be the growth rate of the sequence
$\{a_n=\dim\HF(\phi^n)\}$ for  a monotone  representative $\phi\in\symp^m(M,\omega)$ of $g$:
$$ F^{\infty}(g):=\grow(\dim\HF(\phi^n)) $$
\end{dfn}
\begin{ex}
If $\phi$ is  a non-trivial orientation
preserving periodic diffeomorphism of a compact connected surface $M$
of Euler characteristic $\chi(M) <  0$ , then the  periodicity of the
sequence $\dim\HF(\phi^n)$ implies that  for the  corresponding  mapping class $g$ the  asymptotic invariant
$$ F^{\infty}(g):=\grow(\dim\HF(\phi^n))=1 $$

\end{ex}

\begin{ex} Let $\phi$ be  a monotone  diffeomorphism of finite type of a compact connected surface $M$
of Euler characteristic $\chi(M) <  0$ and $g$ a corresponding algebraically finite mapping class.
Let $U$ be the open regular neighborhood of the $k$ reducing
curves $\gamma_1,......,\gamma_k$ in the Thurston theorem, and $M_j$
be the component of $M\setminus U$.Let $F$ be a fixed point class of $\phi$.
Observe from \cite{JG} that if $F\subset M_j$, then $\ind(F, \phi)=\ind(F,\phi_j)$.So if $F$ is counted in $N(\phi)$ but not counted in $\sum_{j} N(\phi_j)$
, it must intersect $U$. But we see from \cite{JG} that a component
of $U$ can intersect at most 2 essential fixed point classes of $\phi$.
Hence we have $N(\phi)\leq \sum_{j} N(\phi_j)$.For the  monotone  diffeomorphism of finite type $\phi$ maps $\phi_j$ are  periodic.
Applying last  inequality to $\phi^n$ and using  remark   \ref{th:ga} we have

$$
0\leq\dim\HF(\phi^n)= \dim H_*(M^{(n)}_{\id},\p{M^{(n)}_{\id}};\Z_2)+N(\phi^n|M\setminus M^{(n)}_{\id})\leq
$$
$$
\leq \dim H_*(M^{(n)}_{\id},\p{M^{(n)}_{\id}};\Z_2)+N(\phi^n)
$$
$$
\leq \dim H_*(M^{(n)}_{\id},\p{M^{(n)}_{\id}};\Z_2)+ \sum_{j} N((\phi)_j^n) +2k\leq Const
$$
by periodicity of $\phi_j$.
Taking the growth rate in $n$, we get
that asymptotic invariant  $ F^{\infty}(g)=1$.
\end{ex}

\subsection{ Generalised Arnold conjecture. Concluding  remarks \label{sec:problem}}

Let $\phi: M\to M$ be a Hamiltonian symplectomorphism of a compact symplectic
manifold $(M,\omega)$. In the nondegenerate case the Arnold conjecture asserts
that $$\# Fix(\phi) \geq \dim H_*(M,\Q)=\sum_{k=0}^{2n} b_k(M),$$
where $2n=\dim M,  b_k(M)=\dim H_k(M,\Q)$.

The Arnold conjecture was first proved by Eliashberg \cite{eli}
for Riemann surfaces. For tori of arbitrary dimension the Arnold conjecture was proved in the celebrated paper by Conley and Zehnder\cite{conzeh}.
 The most important breakthrough was Floer's
proof of the Arnold conjecture for monotone symplectic manifolds \cite{Floer}.
His proof was based on Floer homology.
His method has been pushed through by Fukaya-Ono\cite{fukono}, Liu-Tian\cite{liutian} and Hofer-Salamon\cite{hofsal} to establish the nondegenerate case of the Arnold conjecture for all symplectic manifolds.

The Hamiltonian symplectomorphism $\phi$ is isotopic to identity map $id_M$.
In this case all fixed points  $\phi$ are in the same Nielsen fixed point class. The  Nielsen number of $\phi$ is 0 or 1 depending on  Lefschetz number is 0 or not. So, the Nielsen number is very weak  invariant   to estimate the number of fixed points of $\phi$ for Hamiltonian symplectomorphism.
From another side, as we saw  in theorem \ref{thm:main},   for the nontrivial periodic
symplectomorphism $\phi$ of a  surface, the Nilsen number of $\phi$ gives an exact estimation from below for the number of nondegenerate fixed points of $\phi$. These  considerations lead us to the following question
\begin{que}
How to estimate the number of nondegenerate fixed points  of general(not necessary Hamiltonian)   symplectomorphism?
\end{que}

\subsubsection{Algebraically finite  mapping class}

 If $\psi$ is a diffeomorphism of finite type of surface $M$ then  $\psi \in \symp^m(M,\omega)$ for some $\psi$-invariant form  $\omega $.
Suppose  that symplectomorphism $\phi$ has only non-degenerate fixed points
and $\phi$ is Hamiltonian isotopic to $\psi$.
 Then $\phi \in \symp^m(M,\omega)$
and $\HF(\phi)$ is isomorphic to $\HF(\psi)$.
So,  from theorem \ref{th:ga}  it follows that
$$
\# Fix(\phi) \geq \dim\HF(\phi)=\dim\HF(\psi)=
$$
$$
=\dim H_*(M_{\psi=\id},\p{M_{\psi=\id}};\Z_2)+N(\psi|M\setminus M_{\psi=\id})=
$$
$$
=\sum_{k=0}^{2} b_k(M_{\psi=\id},\p{M_{\psi=\id}};\Z_2)+N(\psi|M\setminus M_{\psi=\id})
$$
This estimation can be considered as a generalisation of Arnold conjecture
becouse it implies  Arnold conjecture  for   $\psi=id$.
If  $\psi$ is nontrivial orientation preserving  periodic diffeomorphism
then  theorem \ref{thm:main}  implies  an estimation 
$$
\# Fix(\phi) \geq \dim\HF(\phi)=\dim\HF(\psi)= N(\psi)
$$
This estimation can be considered as a  generalisation of  Arnold conjecture for
a nontrivial periodic mapping class.

\subsubsection{Pseudo-Anosov mapping class}

  For pseudo-Anosov  ``diffeomorphism''in given pseudo-Anosov mapping class
 we also have, as in theorems \ref{thm:main},\ref{thm:main0}
and
 \ref{th:ga}, a topological separation of fixed points
 \cite{Th, JG, I}, i.e the Nielsen number of pseudo-Anosov ``diffeomorphism''
 equals to the number of
fixed points and there are  no connecting orbits between them.
But we have the following difficulties.
Firstly, a  pseudo-Anosov ``diffeomorphism'' is a smooth and  a symplectic
automorphism only on the complement of his fixed points set.
Nevertheless, M. Gerber and A. Katok \cite{GK} have found a smooth model for  pseudo-Anosov ``diffeomorphism'' with the same dynamical properties. More precise they have constructed for every  pseudo-Anosov ``diffeomorphism'' $f$ a diffeomorphism $f'$
which is topologically conjugate to $f$ through a homeomorphism isotopic
to identity. Diffeomorphism $f'$ is a symplectomorphism,  it has the same fixed
points as $f$, which are also topologically separated,  and it has the same  Nielsen number as $f$.
Secondly, in the case of a  pseudo-Anosov ``diffeomorphism''  and  it  smooth model, we have  to deal with
fixed points of index $ -p$ where $p>1$. Such  fixed points
are  degenerate from symplectic point of view  and therefore need a local perturbation.

If $\phi$ is monotone  symplectomorphism with nondegenerate fixed points in
given pseudo-Anosov mapping class $ \{ \phi \}=g $  then
$$
\# Fix(\phi) \geq \dim\HF(\phi)=\dim\HF(g)
$$
   To formulate   the generalised Arnold conjecture  in this case  we need to know how
to calculate in classical terms the Floer homology for monotone  symplectomorphism with nondegenerate fixed points  which
represents given pseudo-Anosov mapping class $g$.
For this we need, for example,  to understand the contribution of degenerate  fixed points  of a smooth model of a  pseudo-Anosov ``diffeomorphism'' to the Floer
homology. 
 In \cite{eft} Floer homology were calculated for certain class of pseudo-Anosov
maps which are compositions of positive and negative Dehn twists along loops in $M$
forming a tree-pattern.

\subsubsection{Reducible mapping class.  Generalised Arnold conjecture.}

Suppose now that symplectomorphism $\phi$ has only non-degenerate fixed points
and that  $\phi$ is Hamiltonian isotopic to  a monotone  symplectomorphism $\psi$ in
a  reducible  mapping class $g$ which contains pseudo-Anosov components  $g^i_{pA}, i=1,....,s$( see theorem \ref{thm:thur}).

\begin{conjecture}
 The Floer complex $(\CF(\psi),\p_{J})$
with respect to a generic $J=(J_t)_{t\in\R}$
 splits into the subcomplexes $(\Cc_{fin},\p_{fin})$ and
$(\Cc^i_{pA},\p_{pA}),i=1,...s$, where $\Cc_{fin}$ corresponds to the finite type component $g_{fin}$ of $g$
and $\Cc^i_{pA}, i=1,...s$ correspond to pseudo-Anosov components  $g^i_{pA}, i=1,....s$ of $g$.
The homology of $(\Cc_{fin},\p_{fin})$ is isomorphic to
$\HF(g_{fin})=
H_*(M_{\psi=id},\p{M_{\psi=\id}};\Z_2)\oplus
\Z_2^{N(\psi|M\setminus M_{\psi=id})}
$
by theorem \ref{thm:main0} and 
$$
\# Fix(\phi) \geq \dim\HF(\phi)=\dim\HF(\psi)=\dim\HF(g)=
$$
$$
= \dim\HF(g_{fin})+ \sum_{i=1}^{s}\dim\HF(g^i_{pA})=
$$
$$
=\dim H_*(M_{\psi=\id},\p{M_{\psi=\id}};\Z_2)+N(\psi|M\setminus M_{\psi=\id}) + \sum_{i=1}^{s}\dim\HF(g^i_{pA}) =
$$
$$
=\sum_{k=0}^{2} b_k(M_{\psi=\id},\p{M_{\psi=\id}};\Z_2)+N(\psi|M\setminus M_{\psi=\id})+\sum_{i=1}^{s}\dim\HF(g^i_{pA})
$$
\end{conjecture}

\subsubsection{Concluding remarks}

 Due to P. Seidel  \cite{S1}   $\dim\HF(\phi)$  is a new symplectic invariant
of a four-dimensional symplectic  manifold  with nonzero first Betti number.
This 4-manifold produced from symplectomorphism $ \phi$ by a surgery construction which is a variation of earlier constructions due to McMullen-Taubes,
Fintushel-Stern and J. Smith.We hope that our  asymptotic invariant  and symplectic zeta function $F_{\phi}(z)$  also
give rise to a new  invariants  of contact 3- manifolds
and symplectic 4-manifolds.

\end{document}